\documentclass[12pt]{amsart}

\usepackage{amssymb}

\usepackage{enumerate}

\usepackage{graphicx}

\makeatletter
\@namedef{subjclassname@2010}{%
  \textup{2010} Mathematics Subject Classification}
\makeatother

\newtheorem{lemma}[equation]{Lemma}
\newtheorem{proposition}[equation]{Proposition}
\newtheorem{theorem}[equation]{Theorem}
\newtheorem{corollary}[equation]{Corollary}

\theoremstyle{definition}

\newtheorem{remark}[equation]{Remark}

\numberwithin{equation}{section}

\usepackage[left=3cm,top=3cm,right=3cm]{geometry}

\newcommand{\length}{\operatorname{length}}

\newcommand{\R}{\mathbb{R}}
\newcommand{\N}{\mathbb{N}}
\newcommand{\Z}{\mathbb{Z}}
\newcommand{\Ha}{\mathcal{H}}

\newcommand{\W}{\mathcal{W}}

\newcommand{\sub}{\subset}
\newcommand{\bdry}{\partial}
\newcommand{\vis}{\operatorname{vis}}
\newcommand{\vbdry}{\bdry^{\vis}}

\newcommand{\dist}{\operatorname{dist}}
\newcommand{\diam}{\operatorname{diam}}
\newcommand{\spt}{\operatorname{spt}}

\def\vint{\mathop{\mathchoice%
          {\setbox0\hbox{$\displaystyle\intop$}\kern 0.22\wd0%
           \vcenter{\hrule width 0.6\wd0}\kern -0.82\wd0}%
          {\setbox0\hbox{$\textstyle\intop$}\kern 0.2\wd0%
           \vcenter{\hrule width 0.6\wd0}\kern -0.8\wd0}%
          {\setbox0\hbox{$\scriptstyle\intop$}\kern 0.2\wd0%
           \vcenter{\hrule width 0.6\wd0}\kern -0.8\wd0}%
          {\setbox0\hbox{$\scriptscriptstyle\intop$}\kern 0.2\wd0%
           \vcenter{\hrule width 0.6\wd0}\kern -0.8\wd0}}%
          \mathopen{}\int}

\begin{document}




\title[Fractional Hardy inequalities and visibility]{Fractional Hardy inequalities\\ and visibility of the boundary}



\author[L. Ihnatsyeva]{Lizaveta Ihnatsyeva}   
\address[L.I.]{Aalto University\\
Department of Mathematics and Systems Analysis \\
P.O. Box 11100 \\ 
FI-00076 Aalto, Finland\\
\and
University of Jyvaskyla\\
Department of Mathematics and Statistics \\
P.O. Box 35 (MaD) \\
FI-40014 University of Jyvaskyla, Finland}
\email{lizaveta.ihnatsyeva@aalto.fi}

\author[J. Lehrb\"ack]{Juha Lehrb\"ack}   
\address[J.L.]{University of Jyvaskyla\\
Department of Mathematics and Statistics \\
P.O. Box 35 (MaD) \\
FI-40014 University of Jyvaskyla, Finland}
\email{juha.lehrback@jyu.fi}

\author[H. Tuominen]{Heli Tuominen}   
\address[H.T.]{University of Jyvaskyla\\
Department of Mathematics and Statistics \\
P.O. Box 35 (MaD) \\
FI-40014 University of Jyvaskyla, Finland} 
\email{heli.m.tuominen@jyu.fi}

\author[A. V. V\"ah\"akangas]{Antti V. V\"ah\"akangas}
\address[A.V.V.]{University of Helsinki\\ Department of Mathematics and Statistics\\
P.O. Box 68\\ FI-00014 University of Helsinki, Finland\\ 
\and
University of Jyvaskyla\\
Department of Mathematics and Statistics \\
P.O. Box 35 (MaD) \\
FI-40014 University of Jyvaskyla, Finland} 
\email{antti.vahakangas@helsinki.fi}


\begin{abstract}
We prove fractional order Hardy inequalities on open sets under a combined
fatness and visibility condition on the boundary.
We demonstrate by counterexamples that fatness conditions alone are not  
sufficient for such Hardy inequalities to hold.
In addition, we give a short exposition of various fatness conditions 
related to our main result, and
apply fractional Hardy inequalities in connection to the boundedness of extension
operators for fractional Sobolev spaces.
\end{abstract}

    \maketitle

\section{Introduction}

In this paper, we consider the following
fractional $(s,p)$-Hardy inequalities:
\begin{equation}\label{e.hardy}
\int_G \frac{\lvert u(x)\rvert^p}{\dist(x,\partial G)^{sp}}\,dx
\le c\int_{G} \int_{G}
\frac{|u(x)-u(y)|^p}{|x-y|^{n+sp}}\,dy\,dx\,.
\end{equation}
We say that an open set $G\subset\R^n$ admits an $(s,p)$-Hardy
inequality, for $0<s<1$ and $1<p<\infty$, if there is a
constant $c>0$ such that inequality~\eqref{e.hardy} holds for
every $u\in C^\infty_0(G)$. 
Throughout the paper, we consider only proper open subsets 
$G\subsetneq\mathbb{R}^n$, and therefore the boundary
$\partial G$ is always non-empty.

One of our starting points is a result of Dyda \cite{Dyda}  
that a bounded Lipschitz domain admits an $(s,p)$-Hardy
inequality if and only if $sp>1$.
This can be viewed as an analogue of a result  
of Ne\v cas~\cite{Necas}, stating 
that if $G$ is a bounded Lipschitz domain and $1<p<\infty$, then there is $c>0$
such that the $p$-Hardy inequality
\begin{equation}\label{e.p-hardy}
\int_G \frac{\lvert u(x)\rvert^p}{\dist(x,\partial G)^p}\,dx
\le c \int_G \lvert \nabla u(x)\rvert^p\,dx
\end{equation}
holds for all $u\in C^\infty_0(G)$. For the $p$-Hardy inequality~\eqref{e.p-hardy} it is well understood that
the Lipschitz assumption can be weakened significantly: a canonical sufficient condition is that
the complement of $G$ is $(1,p)$-uniformly fat; we refer here to the works of Lewis~\cite{Lewis1988}
and Wannebo~\cite{Wannebo}. 
On the other hand, it follows easily from the ideas in~\cite{E-HSV} and~\cite{ihnatsyeva1}
that if the complement of an open set $G$ is $(s,p)$-uniformly fat,
then a `global' version of the $(s,p)$-Hardy inequality~\eqref{e.hardy}, namely the inequality
\begin{equation}\label{e.integration}
\int_G \frac{\lvert u(x)\rvert^p}{\dist(x,\partial G)^{sp}}\,dx\le c\int_{\R^n} \int_{\R^n}
\frac{|u(x)-u(y)|^p}{|x-y|^{n+sp}}\,dy\,dx\,,
\end{equation}
holds for every $u\in C^\infty_0(G)$;
here we understand that functions in $C^\infty_0(G)$ are extended as $0$ outside $G$;
see Section~\ref{s.applications} for more details.

The main objective of our present work is to examine how much the
Lipschitz condition can be relaxed without losing the
`localization' in inequality~\eqref{e.hardy}. In other words, we
aim to determine what kind of assumptions on $G$ allow for the
restriction of the integration in the right-hand side
of~\eqref{e.integration} to the set $G\times G$.
Based on a comparison with the $p$-Hardy inequality~\eqref{e.p-hardy}, 
a natural first candidate for such a condition
would be the complement of $G$ being $(s,p)$-uniformly fat.
However, already the example of Dyda~\cite[\S 2]{Dyda} gives an open set
$G$ whose complement is $(s,p)$-uniformly fat for all $1 <
p<\infty$ and $0<s<1$, but where $(s,p)$-Hardy inequalities
fail if $sp\le 1$. In fact, one can simply take $G$ to be the
unit ball of $\R^n$, or any other bounded Lipschitz domain. Notice
that in this case the \emph{boundary} of $G$ is locally $(s,p)$-uniformly
fat only when $sp>1$. Hence one could ask if it is
sufficient for inequality~\eqref{e.hardy} that the boundary is (locally)
$(s,p)$-uniformly fat. Again, it turns out that the answer is
negative, at least in the range $0<sp\le 1$, as we provide in
Section~\ref{s.counterexamples} examples of open sets
whose boundaries are (locally) $(s,p)$-uniformly fat, but which still fail to
admit $(s,p)$-Hardy inequalities.

The obstruction in our examples is that even though the boundary is uniformly fat, 
most of it is not `visible' from within the set.
Hence it seems that a right way to generalize the result of Dyda~\cite{Dyda} beyond Lipschitz domains 
is to apply conditions which combine both fatness and geometry of the boundary.
In our main result, Theorem~\ref{thm:fractional_hardy_visible}, we follow~\cite{kole} and
use a local \emph{visual boundary condition} given in terms of John curves. 
An illustrative consequence of Theorem~\ref{thm:fractional_hardy_visible}
is the following sufficient condition for uniform domains.

\begin{corollary}\label{cor:fractional_hardy_uniform}
Let $0<s<1$ and $1 < p<\infty$ be such that $0<sp<n$. Assume that $G\sub\R^n$
is a (bounded) uniform domain with a (locally) $(s,p)$-uniformly fat boundary.
Then $G$ admits an $(s,p)$-Hardy inequality.
\end{corollary}

Besides bounded Lipschitz domains (for $sp>1$) and domains above
the graphs of Lipschitz functions (also considered by
Dyda~\cite{Dyda}), Corollary~\ref{cor:fractional_hardy_uniform}
covers, for instance, the domain $G\sub\R^2$ bounded by the usual
von Koch snowflake curve of Hausdorff dimension $\lambda={\log
4}/{\log 3}$, or, in fact, any domain inside an analogous curve of
any dimension $\lambda\in (1,2)$. It is indeed well known that
these domains are uniform. Moreover, the boundary of such
a domain $G$ is a $\lambda$-regular set, and therefore the boundary is
locally $(s,p)$-uniformly fat for $sp>2-\lambda$ (cf.\
Proposition~\ref{t.lambda_sets_are_fat}). It then follows from
Corollary~\ref{cor:fractional_hardy_uniform} that $G$ admits an
$(s,p)$-Hardy inequality if (and only if) $sp>2-\lambda$.
(For the converse, we refer to~\cite[\S 2]{Dyda} and the estimates
in Section~\ref{s.sel}.) Nevertheless, without going into the details
in this section, we mention that
Theorem~\ref{thm:fractional_hardy_visible} can be applied to far
more general open sets than just uniform domains. Let us also remark that
in the recent paper~\cite{DV}, related fractional Hardy inequalities are studied using different methods. 
One consequence of the general framework in~\cite{DV} is that in Corollary~\ref{cor:fractional_hardy_uniform} 
the assumption that $G$ is 
uniform can actually be replaced by a plumpness condition for $G$. 
On the other hand, our Theorem~\ref{thm:fractional_hardy_visible}
can be applied to many non-plump cases in which the results in~\cite{DV} do not apply directly,
for instance to domains having outer cusps.

Fractional Hardy inequalities, and their generalizations, have been
enjoying a reasonable amount of interest during the last few
years. One of the reasons for this is that they can be used to deliver spectral information on the
generators of the so-called censored stable processes; see e.g.~\cite{Dyda},
and the references therein, for more details.
In addition to the aforementioned results and references, Loss and
Sloane~\cite{LossSloane2010} established fractional Hardy type
inequalities with sharp constants, but they used different
distance functions in the left-hand side of~\eqref{e.hardy}.
However, in a convex domain $G$ their distance functions are
majorized by $\dist(x,\bdry G)$, and thus one obtains sharp
fractional inequalities~\eqref{e.hardy} for convex domains. Dyda
and Frank~\cite{DydaFrank} further improved the results of Loss
and Sloane into the so-called fractional Hardy--Sobolev--Mazya
inequalities. See
also~\cite{BogdanDyda,Dyda2,FMT,FrankSeiringer,Sloane} for related
results. For one-dimensional fractional Hardy inequalities we
refer to~\cite[Chapter~5]{KufnerPersson} and the references
therein. In the present paper, we restrict ourselves completely to
the case where the boundary of $G$ is `thick', for instance in the
sense of the uniform fatness condition. On the other hand, it is
well-known that Hardy inequalities, in general, can be valid also
when the boundary (or the complement) is `thin' enough. For
fractional Hardy inequalities, the thin case has been examined
systematically in \cite{ihnatsyeva2, DV}; the work of Dyda~\cite{Dyda} contains some particular instances of thin
complements as well.

The outline of the rest of the paper is as follows: In Section~\ref{s.notation} 
we provide definitions, notation, and other
basic tools; for instance, the definition of uniform fatness, based on
Riesz capacities, and the different John and uniformity conditions
related to the visibility can be found here. Section~\ref{s.comparison} 
contains comparison results for various
density conditions. In particular, we relate conditions given in
terms of capacities (as in uniform fatness) to conditions for
Hausdorff contents. Our main result,
Theorem~\ref{thm:fractional_hardy_visible}, is stated and proved
in Section~\ref{s.proof_main}. The proof is based on a combination
of `pointwise Hardy' techniques developed in~\cite{kole,KLT,LeLip}
and maximal function arguments, similar to those
in~\cite{E-HSV,ihnatsyeva1}. In Section~\ref{s.counterexamples} we construct
the counterexamples showing that $(s,p)$-uniform fatness
of the boundary $\bdry G$ does not suffice for $G$ to admit an
$(s,p)$-Hardy inequality. We close the paper in Section~\ref{s.applications}
 with some applications and additional results
related to the extension of functions in the fractional Sobolev
space $ W^{s,p}(G)$ and to the removability of sets with respect
to $(s,p)$-Hardy inequalities.

\section{Notation and preliminaries}\label{s.notation}

\subsection{Basic notation and Whitney cubes}
We follow the standard convention that the letters $C$ and $c$ 
denote positive constants whose values are not necessarily the 
same at each occurrence. If there exist constants $c_1,c_2>0$ such that $c_1\,a\le
b\le c_2 a$, we sometimes write $a\simeq b$ and say that $a$ and $b$ are comparable.

When $E\subset\R^n$,
the characteristic function of $E$ is denoted by $\chi_E$, the boundary of $E$ is written
as $\partial E$, and $\lvert E\rvert$ is the $n$-dimensional Lebesgue measure of $E$.
The complement of a set $G\subset\R^n$ is denoted by $G^c$.
The integral average of a locally integrable function $f$ over a bounded set $E$ with a positive measure 
is written as
\[
f_E := \vint_E f\, dx := \frac1{\rvert E \lvert}\int_E f\, dx\,.
\]
The support of a function $f\colon\R^n\to \mathbb{C}$ is denoted by
$\mathrm{spt}(f)$, and it is the closure of the set  $\{x\,:\,f(x)\not=0\}$ in $\R^n$. If $\mu$ is a Borel measure in $\R^n$,
the support of $\mu$, denoted by $\mathrm{spt}(\mu)$,
is the smallest closed set $E$ such that $\mu(\R^n\setminus E)=0$.

The open ball centered at $x\in\R^n$ and of radius $r>0$ is denoted by
$B(x,r)$, and $Q$ always denotes a cube in $\R^n$ with sides parallel to
the coordinate axes; we write $x_Q$ for the center of the cube $Q$ and
$\ell(Q)$ for its side length. For $L>0$, we write $L Q$ for the
dilated cube with side length $L \ell(Q)$ and center $x_Q$.

The family of closed dyadic cubes is denoted by $\mathcal{D}$, and
$\mathcal{D}_j$ is the family of dyadic cubes of side length
$2^{-j}$, $j\in\Z$. For a proper open set $G$, we fix its
Whitney decomposition $\W=\W(G)\subset\mathcal{D}$, and write
\[ \W_j=\W_j(G) := \W(G)\cap\mathcal{D}_j \] for $j\in\Z$.
Recall that $G=\cup_{Q\in\mathcal{W}(G)} Q $ and each Whitney cube $Q\in\W(G)$ satisfies
\begin{equation}\label{dist_est}
  \diam(Q)\le \dist(Q,\partial G)\le 4\diam(Q)\,.
\end{equation}
For the construction and other properties of Whitney cubes we refer to \cite[VI.1]{S}.

\subsection{Riesz potentials and uniform fatness}\label{s.riesz}

The Riesz $s$-potentials, $0<s<n$, of a measurable
function $f$ and a Borel measure $\mu$ on $\R^n$ are given by, respectively,
\[
\mathcal{I}_{s} f(x) =\int_{\R^n} \frac{f(y)}{|x-y|^{n-s}}\,dy\quad \text{ and }\quad\mathcal{I}_s \mu(x)
=\int_{\R^n} \frac{1}{|x-y|^{n-s}}\,d\mu(y)\,.
\]
If $0 < sp <n$ and $1<p<\infty$, the $(s,p)$-outer capacity of a set $E$ in $\R^n$  is 
\[
R_{s,p}(E)=\inf\big\{ \lvert\lvert f\rvert\rvert_p^p\,:\, f\ge 0\text{ and }\mathcal{I}_{s}f\ge1\text{ on }E\big\}\,.
\]
Following \cite{Lewis1988}, we say that a set $E$ is locally $(s,p)$-uniformly fat for $0<sp<n$, $1<p<\infty$,
if there exist positive constants $r_0$ and $\sigma$ such that
\begin{equation}\label{eq.fatness}
R_{s,p}(B(x,r)\cap E)\ge \sigma r^{n-s p}
\end{equation}
for every $x\in E$ and $0<r<r_0$.
If inequality \eqref{eq.fatness} holds for every
$x\in E$ and every $r>0$, we say that $E$ is $(s,p)$-uniformly fat.

\subsection{Hausdorff measures and regular sets}
The \emph{$\lambda$-Hausdorff content} of a set $E \sub\R^n$ is 
\[
\Ha^\lambda_\infty(E)=\inf\bigg\{\sum_{i=1}^{\infty}r_i^\lambda :
 E\sub\bigcup_{i=1}^\infty B(x_i,r_i),\ r_i>0 \bigg\},
\]
where we may assume that $x_i\in E$, since this increases the infimum at most by a constant factor.
Moreover, as is easily seen, we may allow finite coverings in the infimum above.
The $\lambda$-Hausdorff measure is denoted by $\mathcal{H}^\lambda$, for the definition we refer to \cite[\S 4]{Mattila}.
 Recall that the \emph{Hausdorff dimension} of $E$ is the number
\[
\dim_{\mathcal{H}}(E)=\inf\{\lambda>0 : \Ha^\lambda_\infty(E)=0\}= \inf\{\lambda>0 : \Ha^\lambda(E)=0\}.
\]

Let $0<\lambda\le n$.
A closed set $E\subset \R^n$ is said to be an {\em (Ahlfors) $\lambda$-regular set},
or sometimes simply a $\lambda$-set, if there is a constant $C>1$ such that
\begin{equation}\label{e.lambda_set}
C^{-1}r^\lambda \le \mathcal{H}^\lambda(B(x,r)\cap E)\le Cr^\lambda
\end{equation}
for every $x\in E$ and all $0<r < \diam(E)$.

\subsection{John domains and visual boundary}\label{sect: visual}

We say that a domain (an open and connected set) $G\sub\R^n$ is a \emph{$c$-John domain}, $c\geq 1$, with center point $x_0$,
if for every $x\in G$ there exists a curve (called a John curve) $\gamma\colon [0,\ell]\to G$,
parameterized by arc length, such that $\gamma(0)=x$, $\gamma(\ell)=x_0$, and
\begin{equation}\label{eq: John}
\dist(\gamma(t),\bdry G)\geq\tfrac 1 c t
\end{equation}
for each $t\in[0,\ell]$.
If $G$ is a $c$-John domain with center point $x_0$, then
\begin{equation}\label{e.john_inclusion}
G\sub B\big(x_0, c\,\dist(x_0,\bdry G)\big)\,;
\end{equation}
in particular, $G$ is bounded.
Also, if $ G$ is a $c$-John domain, then for each
$w\in\bdry G$ there is a curve $\gamma\colon [0,\ell]\to G\cup\{w\}$
joining $w$ to $x_0$ and satisfying \eqref{eq: John}.
We say in this case, too, that $\gamma$ joins $w$ to $x_0$ {in} $ G$.

When $ G\sub\R^n$ is an open set, $x\in G$, and $c\geq 1$ is a constant,
we define a subdomain $ G_{x,c}$ by
\[
 G_{x,c}= 
  \bigcup\,\{U\sub G :
  U \text{ is a } c\text{-John domain with center point } x\}.
\]
Then clearly $\emptyset\neq G_{x,c}\sub G$ and $ G_{x,c}$ is also a $c$-John domain
with center point $x$. Following~\cite{kole}, we say that the set
\[
\bdry^{\vis}_{x,c} G := \bdry G\cap\bdry G_{x,c}
\]
is the \emph{$c$-visual boundary of $ G$ near} $x$.

A domain $ G\sub\R^n$ is a \emph{uniform domain} if there is a
constant $C\geq 1$ such that each pair of points $x,y\in G$ can be
joined by a curve $\gamma\colon [0,\ell]\to G$, parameterized by arc
length, so that $\ell\leq C \lvert x-y\rvert$ and $\dist(z,\bdry
G)\geq \frac 1 C \min\{\lvert z-x\rvert, \lvert z-y\rvert\}$ for
each $z\in\gamma$\,.
Every bounded uniform domain is also a $c$-John domain for some $c\geq 1$.

\section{Comparison results}\label{s.comparison}

In this expository section, we study the connections between
uniform fatness and thickness conditions formulated in
terms of Hausdorff contents; thereby we clarify the relations of
the conditions which commonly appear in connection to our main theorem.

More precisely, in our main result, Theorem~\ref{thm:fractional_hardy_visible}, we assume that the visual
boundary is uniformly large near each point $x\in G$, in the sense
of the $\lambda$-Hausdorff content for an appropriate exponent $\lambda>0$,
as follows:
\begin{equation}\label{eq:main cond_intro}
  \Ha^\lambda_\infty\big(\bdry^{\vis}_{x,c} G\big)
            \geq C_0 \dist(x,\bdry  G)^\lambda,
\end{equation}
where $C_0>0$ and $c\ge 1$ are independent of $x$.
A closely related condition is an {\it inner boundary density
condition} with an exponent $\lambda$, which is satisfied by an open set
$G\subset\R^n$ if there exists a constant $C>0$ such that
\begin{equation}\label{eq: boundarydensity}
\mathcal{H}^\lambda_\infty\big(\partial G \cap B(x,2\dist(x,\partial G))\big) \ge C\dist(x,\partial G)^\lambda
\end{equation}
for all $x\in G$ (cf.~\cite{LePoint}).

Our main assumption \eqref{eq:main cond_intro} is stronger than 
the inner boundary density condition~\eqref{eq: boundarydensity}
with the same exponent $\lambda$. Indeed, one may apply
relation \eqref{e.john_inclusion} and the observation that an open
set $G\sub\R^n$ satisfies condition \eqref{eq:main cond_intro} if
and only if for each $x\in G$ there exists \emph{some} $c$-John
domain $U_x\sub G$, with center point $x$, such that
\[
  \Ha^\lambda_\infty(\bdry U_x\cap\bdry G)
            \geq C_0 \dist(x,\bdry G)^\lambda\,.
\]
On the other hand, there is a wide class of domains for which a
converse is true as well. For instance, a uniform domain $G$
satisfies condition~\eqref{eq:main cond_intro} if and
only if $G$ satisfies condition~\eqref{eq: boundarydensity}. In
fact, for uniform domains the visual boundary near $x$ coincides
with the  boundary (near  $x$), as the following adaptation of
\cite[Proposition~4.3]{kole} shows.
\begin{proposition}\label{prop:unif}
Assume that $G\sub\R^n$ is a uniform domain. Then there exists a
constant $c\geq 1$, depending only on $n$ and the constant in the
uniformity condition for $G$, such that for all $x\in G$
\[\bdry G\cap B(x,2\dist(x,\bdry G)) \sub \vbdry_{x,c}G\,.\]
\end{proposition}

The inner boundary density condition~\eqref{eq: boundarydensity} is satisfied with an exponent 
$\lambda>n-sp$, e.g.\ in the
case of a (bounded) open set with a (locally) $(s,p)$-uniformly fat
boundary, where $1<p<\infty$ and $0<sp<n$. This fact is a consequence of 
Proposition~\ref{prop:fatness implies density} below. Nevertheless, we remark that
the fatness of the \emph{boundary} is not necessary for either of the
conditions~\eqref{eq:main cond_intro} and~\eqref{eq: boundarydensity} to hold,
as domains with outward cusps show. 

More generally, the following Theorem~\ref{t.equiv}
gives a precise connection between the inner
boundary density condition and the uniform fatness of the
complement of an open set; as indicated, the case $0<sp\le
1$ is of particular interest. The theorem mostly restates, and
also slightly extends, some of the results obtained in~\cite{LePoint}.

\begin{theorem}\label{t.equiv}
Let $1<p<\infty$ and  $0<sp<n$, and consider the following conditions
for an open set $G\subset\R^n$: 
\begin{itemize}
\item[(1)] the complement $G^c$ is $(s,p)$-uniformly fat; 

\item[(2)] there exists $n-sp<\lambda\le n$ and $C>0$ such that, for all $r>0$ and 
for every $x\in G^c$,
$\mathcal{H}^\lambda_\infty(B(x,r)\cap G^c) \ge Cr^\lambda$\,;
\item[(3)] $G$ satisfies~\eqref{eq: boundarydensity}
with an exponent $n-sp< \lambda \le n$.
\end{itemize}
Then conditions (1) and (2) are (quantitatively) equivalent, and 
condition (3) implies (1) and (2). Moreover, for $sp>1$, all
of the conditions (1)--(3) are equivalent; here the assumption
$sp>1$ cannot be relaxed in general.
\end{theorem}

\begin{proof}
The equivalence (1) $\Leftrightarrow$ (2) is a consequence of
Propositions~\ref{prop:fatness implies density} and~\ref{p.density_implies_fatness} proved in 
Sections~\ref{s.thick_1} and~\ref{s.thick_2} below. We remark that these propositions are
most likely known to the experts, and they are based on ideas in \cite[\S 5]{AH},
but we include some details for the sake of exposition.

The implication (3) $\Rightarrow$ (2) is proved in~\cite[pp.\ 2197--2198]{LePoint}.
Furthermore, if $sp>1$, then the
converse implication follows by a careful inspection of the proofs
in~\cite{LePoint}; note that in~\cite{LePoint} the claims are formulated for
domains, but they actually hold for all open sets.
The ball $G=B(0,1)$ serves as a counterexample,
showing that the assumption $sp>1$ is necessary for the
implication (2) $\Rightarrow$ (3) to hold in general (recall that
$\partial G$ has Hausdorff dimension $n-1$, hence
$\mathcal{H}^\lambda_\infty(\partial G)=0$ if $n-1<\lambda \le
n$). 
 \end{proof}

\subsection{Uniform fatness $\Rightarrow$ thickness}\label{s.thick_1}
(By `thickness'  we refer to density in terms
of Hausdorff contents, see condition (2) in Theorem~\ref{t.equiv}.)

We begin with an easy bound for Riesz capacities in terms of Hausdorff contents:

\begin{lemma}\label{lemma:Riesz bounds Hausdorff}
If $1<p<\infty$ and $0 < sp <n$, then there is a constant $C>0$
such that $R_{s,p}(E) \le C \Ha^{n-sp}_\infty(E)$
 for all sets $E\sub \R^n$.
\end{lemma}

\begin{proof}
 Let $E\sub\bigcup_i B_i$, where $B_i=B(x_i,r_i)$ with $x_i\in\R^n$, $r_i>0$.
 By the monotonicity and the subadditivity of the Riesz capacity, we have
 \[
 R_{s,p}(E)\leq R_{s,p}\Big(\bigcup_i B_i\Big)\leq  \sum_i R_{s,p}(B_i)\,.
 \]
 Since $R_{s,p}(B_i)\leq C r_i^{n-sp}$ for all balls, \cite[Proposition~5.1.2]{AH},
 the claim follows by taking the infimum over all
 such covers of $E$.
\end{proof}

The following deep theorem from~\cite{Lewis1988} states that local
uniform fatness of closed sets is a self-improving property.

\begin{theorem}\label{lewis}
Let $0< sp<n$ and $1<p<\infty$. Assume that $E\sub\R^n$ is closed and
locally $(s,p)$-uniformly fat (with constant $\sigma>0$). Then there are constants
$\varepsilon,\sigma_1>0$, depending only on $s,p,n,\sigma$, such
that $R_{\beta,q} (B(x,r)\cap E)  \ge \sigma_1 r^{n-\beta q}$
whenever $x\in E$, $0<r<r_0$, and $sp-\varepsilon<\beta q<sp$.
\end{theorem}

In other words, a closed and
locally $(s,p)$-uniformly fat set $E$ is actually 
$(\beta,q)$-locally uniformly fat also for
$sp-\varepsilon<\beta q \le sp$. Since the constants $\varepsilon$
and $\sigma_1$ in the formulation of Theorem \ref{lewis} are
independent of the parameter $r_0$, we see that uniform
fatness of closed sets is self-improving as well.
With the help of the self-improvement we obtain the following
proposition, which yields in particular the implication (1)$\Rightarrow$(2)
of Theorem~\ref{t.equiv}.

\begin{proposition}\label{prop:fatness implies density}
 Let $1<p<\infty$ and $0 < sp <n$.
 Assume that $E\sub\R^n$ is closed and $(s,p)$-uniformly fat. Then there exist
 $n-sp< \lambda \le n$ and $C>0$ such that
 \[
 \Ha^\lambda_\infty(B(x,r)\cap E)\ge C r^\lambda
 \]
 for every $x\in E$ and for all $r>0$.
\end{proposition}

\begin{proof}
 By the self-improvement, there is $1<q<p$ such that $E$ is  $(s,q)$-uniformly fat.
 It follows from the definition of $(s,q)$-uniform fatness
 and Lemma \ref{lemma:Riesz bounds Hausdorff} that
 \[
 r^{n-sq} \le \sigma_1^{-1} R_{s,q}(B(x,r)\cap E) \le C \Ha^{n-sq}_\infty(B(x,r)\cap E)
 \]
 for every $x\in E$ and all $r>0$. Hence the claim follows with $\lambda = n - sq > n - sp$.
\end{proof}

Note that if $E$ is locally $(s,p)$-uniformly fat, then the claim 
of Proposition~\ref{prop:fatness implies density} holds for all $0<r<r_0$. 

\subsection{Thickness $\Rightarrow$ uniform fatness}\label{s.thick_2}
Suppose $E\subset \R^n$, $1<p<\infty$, and $0<sp<n$. As in~\cite{Mattila}, 
$\mathcal{M}(E)$ denotes the set of Radon
measures with compact support satisfying
$\operatorname{spt}(\mu)\subset E$ and $0<\mu(\R^n)<\infty$. 
Below we use the following variant of
Frostman's lemma from \cite[Theorem 8.8]{Mattila}; see also~\cite{Carleson}.

\begin{lemma}\label{l.frostman}
Let $E$ be a Borel set in $\R^n$ and let $\lambda>0$. Then $\mathcal{H}^\lambda(E)>0$
if and only if there exists $\mu\in \mathcal{M}(E)$
such that $\mu(B(x,r)) \le r^\lambda$ for $x\in\R^n$ and
$r>0$. Moreover, we can find $\mu$ so that
$\mu(E)\ge C_n\,\mathcal{H}^{\lambda}_\infty(E)$ where
$C_n>0$ depends only on $n$.
\end{lemma}

The Wolff potential of a measure $\mu\in \mathcal{M}(\R^n)$ at $ y\in\R^n$
is 
\[
\dot{W}^{\mu}_{s,p}(y) = \int_0^\infty
\bigg(\frac{\mu(B(y,t))}{t^{n-sp}}\bigg)^{p'-1} \frac{dt}{t}\,,
\,\qquad p'=p/(p-1)\,.
\]
By
Wolff's inequality, there is $c>0$ such that
\begin{equation}\label{e.wolff_application}
\sup \mu(E)\le c R_{s,p}(E)\,,
\end{equation}
where the supremum is taken over $\mu\in\mathcal{M}(E)$ satisfying
$\dot{W}^{\mu}_{s,p}(y)\le 1$ for all $y\in \R^n$.

\begin{remark}
Although inequality \eqref{e.wolff_application} is well known and
widely used, it is worthwhile to sketch a proof. Here it is
convenient to use a dual definition for the Riesz capacity of a
compact set $K\subset E$, statement (A) in \cite[p.~177]{Lewis1988}:
\[
R_{s,p}(K) = \sup\{\upsilon(K)^p \,:\,
\,\mathrm{spt}(\upsilon)\subset K \text{ and }\lVert \mathcal{I}_s
\upsilon\rVert_{p'} \le 1\}\,.
\]
Let $\mu\in\mathcal{M}(E)$ be a measure such that
$\dot{W}^\mu_{s,p}(y)\le 1$ for every $y\in \R^n$. Write
$K=\spt(\mu)$. By Wolff's inequality, \cite[Theorem 4.5.4]{AH},
there is  $A>0$,
depending on $s$, $p$, and $n$, such that
\begin{equation}\label{e.admissible}
 \int_{\R^n} \lvert \mathcal{I}_s \mu\rvert^{p'}\,dx  = \int_{\R^n} ( \mathcal{I}_s \mu)^{p'}\,dx
\le A \int_{\R^n} \dot{W}^{\mu}_{s,p}\,d\mu\le A\mu(K)\,.
\end{equation}
Define $\upsilon := A^{-1/p'} \mu(K)^{-1/p'}\mu\in\mathcal{M}(K)$.
By inequality~\eqref{e.admissible}, the measure $\upsilon$ is
admissible for the dual definition of capacity. Since $K\subset
E$, we have that
\[
R_{s,p}(E)\ge R_{s,p}(K)\ge \upsilon(K)^p = A^{-p/p'}
\mu(K)^{-p/p'}\mu(K)^p = A^{-p/p'} \mu(E)\,.
\]
Taking the supremum over measures $\mu$ yields inequality
\eqref{e.wolff_application} with the implied constant
$A^{p/p'}$. 
\end{remark}

The following proposition shows that thickness implies uniform
fatness; in particular, (2)$\Rightarrow$(1) holds in
Theorem~\ref{t.equiv}. For a closely related comparison theorem, see
\cite[Corollary~5.1.14]{AH}.

\begin{proposition}\label{p.density_implies_fatness}
Let $s,p,\lambda$ be such that $1<p<\infty$, $0<sp<n$, and
$n-sp<\lambda \le n$. Let $E\subset \R^n$ be a Borel set such that
for a constant $\sigma>0$,
\begin{equation}\label{e.estimate}
\mathcal{H}^\lambda_\infty (B(x,r)\cap E) \ge \sigma r^\lambda
\end{equation}
for all $x\in E$ and $r>0$. Then $E$ is $(s,p)$-uniformly fat.
\end{proposition}

\begin{proof}
Let $x\in E$, $r>0$, and let
\[
E_{x,r}= B(x,r)\cap E\,.
\]
By condition \eqref{e.estimate}, we find that
$\mathcal{H}^\lambda(E_{x,r})>0$. Frostman's Lemma 
\ref{l.frostman} yields a
measure $\widetilde{\mu}\in \mathcal{M}(E_{x,r})$ satisfying
$\widetilde{\mu}(B(y,t))\le t^\lambda$ for each $y\in \R^n$ and
$t>0$. Moreover,
\begin{equation}\label{e.frost_ineq}
\widetilde{\mu}(E_{x,r})\ge
C_n\mathcal{H}^\lambda_\infty(E_{x,r})\,,
\end{equation}
where $C_n>0$ depends only on $n$.

Now, for a fixed $y\in\R^n$, we write
\begin{align*}
\dot{W}^{\widetilde{\mu}}_{s,p}(y) = \bigg\{ \int_0^r  +
\int_r^\infty\bigg\}\bigg(\frac{\widetilde{\mu}(B(y,t))}{t^{n-sp}}\bigg)^{p'-1}
\frac{dt}{t} =: A+ B\,.
\end{align*}
We first estimate term $A$ using the properties of the Frostman
measure $\widetilde{\mu}$:
\begin{align*}
A &= \int_0^r \bigg(\frac{\widetilde{\mu}(B(y,t))}{t^{n-sp}}\bigg)^{p'-1} \frac{dt}{t} \\
&\le \int_0^r  t^{(\lambda-n+sp)(p'-1) - 1} \,dt =
\frac{r^{(\lambda-n+sp)(p'-1)}} {(\lambda - n +sp)(p'-1)}\,.
\end{align*}
Next term $B$ is considered. We begin with a preliminary
observation: for $y\in \R^n$ and $t>0$,
\[
\widetilde{\mu}(B(y,t))= \widetilde{\mu}(B(y,t)\cap B(x,r)) \le
\widetilde{\mu}(B(x,r))\le r^\lambda\,.
\]
This follows from the fact that $\widetilde{\mu}$ is supported
inside $E_{x,r}\subset B(x,r)$. Thus we obtain
\begin{align*}
B \le  r^{\lambda(p'-1)} \int_r^\infty t^{(sp-n)(p'-1)-1}\,dt =
\frac{r^{\lambda(p'-1)+(sp-n)(p'-1)}}{(n-sp)(p'-1)} =
\frac{r^{(\lambda-n +sp)(p'-1)}}{(n-sp)(p'-1)}\,.
\end{align*}
The above estimates show that $\dot{W}^{\widetilde{\mu}}_{s,p}(y)
= A + B\le \kappa r^{(\lambda-n+sp)(p'-1)}$, where the positive
constant $\kappa$ depends on $n$, $s$, $p$, and $\lambda$.

Let us define $\mu := \kappa^{-1/(p'-1)} r^{n-sp-\lambda}
\widetilde\mu$. Then $\mu\in \mathcal{M}(E_{x,r})$ and, for every
$y\in \R^n$,
\[
\dot{W}^{\mu}_{s,p}(y) =
\kappa^{-1}r^{(n-sp-\lambda)(p'-1)}\dot{W}^{\widetilde{\mu}}_{s,p}(y)\le
1\,.
\]
Hence, by 
assumption \eqref{e.estimate} and inequalities
\eqref{e.frost_ineq} and \eqref{e.wolff_application},
\[
\sigma r^{\lambda} \le\mathcal{H}^\lambda_{\infty}(E_{x,r}) \le
C_n^{-1} \widetilde\mu(E_{x,r}) \le c C_n^{-1}\kappa^{1/(p'-1)}
r^{\lambda-n+sp}R_{s,p}(E_{x,r})\,.
\]
After a simplification of the exponents 
we find that
\[
C_{\sigma,\lambda,n,s,p} r^{n-sp}
\le R_{s,p}(E_{x,r})=R_{s,p}(B(x,r)\cap E)\,.
\]
This concludes the proof.
\end{proof}

Note again that if the thickness condition~\eqref{e.estimate} holds for all $x\in E$ and every $0<r<r_0$,
then $E$ is locally $(s,p)$-uniformly fat.

\subsection{$\lambda$-regular sets and uniform fatness}
The following result relating $\lambda$-regular sets and uniform fatness
is  useful from the viewpoint of applications.

\begin{proposition}\label{t.lambda_sets_are_fat}
Let $s,p,\lambda$ be such that $1<p<\infty$, $0<sp<n$, and
$n-sp<\lambda\le n$. If $E\subset \R^n$ is an unbounded
$\lambda$-regular set, then $E$ is $(s,p)$-uniformly fat. On the other
hand, if $E\subset \R^n$ is a bounded $\lambda$-regular set, then $E$ is
locally $(s,p)$-uniformly fat.
\end{proposition}

\begin{proof}
Let us consider the unbounded case; the bounded case is similar.
Let $x\in E$, $r>0$, and let $E_{x,r}=\bar B(x,r/2)\cap E$. Since
$\lambda$-regular sets are closed (by definition), the set $E_{x,r}$ is compact, and we
may consider the measure
\[
\widetilde{\mu} = \mathcal{H}^\lambda\lvert_{E_{x,r}} \in
\mathcal{M}(E_{x,r})\,.
\]
For the last relation, we refer to \cite[p. 57]{Mattila}. By 
the regularity condition~\eqref{e.lambda_set},
for each $y\in \R^n$ and for all $0<t\le r$ we have
\[\widetilde{\mu}(B(y,t))\le C t^\lambda\] 
and, moreover,
$\widetilde{\mu}(\R^n)=\widetilde{\mu}(E_{x,r})\simeq r^\lambda$.
Hence, proceeding as in the proof of Proposition~\ref{p.density_implies_fatness}, we obtain
\[
r^{n-sp}\le c R_{s,p}(E_{x,r}) \le c R_{s,p}(B(x,r)\cap E)\,.
\]
This concludes the proof, as $x\in E$ and $r>0$ are arbitrary.
\end{proof}

\section{Main Theorem}\label{s.proof_main}

The following sufficient condition for fractional Hardy inequalities is 
the main result of this paper.

\begin{theorem}\label{thm:fractional_hardy_visible}
Let $0<s<1$ and $1<p<\infty$ satisfy $0<sp<n$, and let $G\sub\R^n$
be an open set. Assume that
there exists $n - sp<\lambda \le n $ and $C_0>0$, $c\ge 1$ such that
\begin{equation}\label{eq:main cond}
\Ha^\lambda_\infty\big(\bdry^{\vis}_{x,c} G\big)
\geq C_0 \dist(x,\bdry  G)^\lambda
\end{equation}
for all $x\in G$. Then $G$ admits an $(s,p)$-Hardy inequality.
\end{theorem}

The proof of Theorem~\ref{thm:fractional_hardy_visible} is based upon a general scheme, built in
\cite{E-HSV,ihnatsyeva1}, in combination with visual boundary
and pointwise Hardy techniques, developed in~\cite{kole,KLT,LeLip}.
For the application of the latter in the present setting, we need the following fractional 
Poincar\'e-type inequalities: 

\begin{lemma}\label{lemma:poincare_cube}
Let $Q$ be a cube in $\R^n$, $n\ge 2$, and let 
$p,q\in [1,\infty)$, $\beta \in (0,1)$, and
$1/q-1/p<\beta /n$. Then the following fractional
$(p,q,\beta)$-Poincar\'e inequality holds 
for every $u\in L^q(Q)$:
\[
\int_{Q} |u(x)-u_Q|^p\,dx \le c|Q|^{1+p\beta/n-p/q} \bigg(\int_{Q}\int_{Q}
\frac{|u(x)-u(y)|^q}{|x-y|^{n+\beta q}}\,dy\,dx\bigg)^{p/q}
\]
Here the constant $c>0$ is independent of $Q$ and $u$.
\end{lemma}

Regarding the proof of Lemma~\ref{lemma:poincare_cube}, in the Sobolev--Poincar\'e case 
$p>q$ we apply scaling and translation invariance, and
thereby reduce to the unit cube $[-1/2,1/2]^n$.  For the
unit cube, the claim follows from~\cite[Remark 4.14]{H-SV}.
If $p\le q$, we apply the previous case and H\"older's inequality.

The following key estimate yields a connection between the size of the
visual boundary and the double integrals appearing in
the right-hand sides of the fractional Hardy inequalities.

\begin{lemma}\label{lemma:Hausdorff estimate}
Let $G\sub\R^n$ be an open set.
Assume that $0<\beta<1$, $1\le q<\infty$, and $n-\beta q < \lambda \le n$.
Fix $c\geq 1$ and let $Q\in\W(G)$.
Then there exist constants $L=L(c,n)\geq 1$
and $C=C(n,c,\lambda,q,\beta)>0$ (both independent of $Q$), such that
\begin{equation}\label{eq:bound on Ha}
 \Ha_{\infty}^\lambda (\bdry^{\vis}_{x_Q,c}G)|u_Q|^{q} \leq
     C  \ell(Q)^{\lambda-n+\beta q} \int_{LQ\cap G}\int_{LQ\cap G}
        \frac{|u(x)-u(y)|^q}{|x-y|^{n+\beta q}}\,dy\,dx
\end{equation}
for every $u\in C_0^\infty(G)$.
\end{lemma}

\begin{proof}
Throughout the proof, $C$ denotes a positive constant, whose value may change from one occurrence
to another, and which depends at most
on $n$, $c$, $\lambda$, $q$, and $\beta$. 

Let $u\in C_0^\infty(G)$. If $|u_Q|=0$, the claim is trivial, and thus we may assume $|u_Q|\neq 0$.
Moreover, by homogeneity, we may assume $|u_Q|=1$ (otherwise just consider the function $v=u/|u_Q|$).

Let $w\in \vbdry_{x_Q,c}G$, and let $\gamma$ be a $c$-John curve connecting $w$ to the center point
$x_Q$ in $ G_{x_Q,c}\subset G$.
We construct a chain of $\mathcal{W}(G)$-cubes $C(Q,w)=(Q_0, Q_1, \ldots)$
joining  $Q_0=Q$ to $\omega$ as follows. First, let us set $t_0=\ell$, where $\gamma(\ell)=x_Q$.
Then, in the inductive stage, if cubes $Q_0,\ldots, Q_i$ and numbers $t_0,\ldots,t_i$ are chosen,
we let $0<t_{i+1} < t_i$ be the smallest number for which
 $\gamma(t_{i+1})$ belongs to a cube
 $Q'\in\mathcal{W}(G)$ which intersects $Q_i$. We set $Q_{i+1}=Q'$.
Then, by the construction, $\gamma(t_i)\in Q_i$ and 
$\gamma(t)$, for $0<t<t_{i+1}$, does not belong to any cube
$Q''\in\mathcal{W}(G)$ intersecting $Q_i$.
In particular,
\[
\lvert \gamma(t_i)-\gamma(t_{i+1})\rvert \ge \ell(Q_i)/5,
\]
since, by~\eqref{dist_est}, $\ell(Q_i)/5\le \ell(Q'')$ 
for all such cubes $Q''$.
By the arc-length parametrization of $\gamma$ we thus obtain
 \begin{equation}\label{e.clarify_1}
 t_i - t_{i+1} \ge \lvert \gamma(t_i)-\gamma(t_{i+1})\rvert
 \ge \ell(Q_i)/5 \qquad \text{for all } i\ge 0\,.
\end{equation}



Next we show that  
\begin{equation}\label{eq:basic}
\lim_{i\to \infty} u_{Q_i} = u(w)=0.
\end{equation}
To this end, for $i\ge 0$ we first estimate
\begin{equation}\label{eq:useful}
\begin{split}
\diam(Q_i)&\le \dist(Q_i,\partial G)\le \lvert \gamma(t_i)-w\rvert = \lvert \gamma(t_i)-\gamma(0)\rvert\\
&\le t_i \le c\dist(\gamma(t_i),\partial G) \le 5c\diam(Q_i)\,.
\end{split}
\end{equation}
By inequalities \eqref{e.clarify_1} and \eqref{eq:useful}, we find that
\[t_{i}-t_{i+1}\ge \ell(Q_i)/5\ge t_i/(25 c \sqrt n)\,,\] where $C(n,c)=(25 c \sqrt n)^{-1}\in (0,1)$, 
and thus
\begin{equation}\label{eq:useful2}
t_i\le (1-C(n,c))^i t_0\xrightarrow{i\to\infty} 0\,.
\end{equation}
To conclude the proof of~\eqref{eq:basic},
it suffices to observe that
$Q_i\subset G\cap B(w, 3t_i)$;  indeed, by 
the arc-length parametrization of $\gamma$ and estimate~\eqref{eq:useful}
we obtain, for every 
$y\in Q_i$, that
\[
\lvert y-w\rvert \le \lvert y-\gamma(t_i)\rvert 
+ \lvert \gamma(t_i)-\gamma(0)\rvert \le \diam(Q_i)+t_i \le 2t_i < 3t_i\,.
\]

Observe also that~\eqref{eq:useful2} implies 
\begin{equation}\label{eq:delta sum}
  \sum_{i=0}^\infty \ell(Q_i)^\delta \leq \sum_{i=0}^\infty t_i^\delta \le t_0^\delta \sum_{i=0}^\infty
  (1-C(n,c))^{i\delta }
  \le  C(n,c,\delta)\ell(Q_0)^\delta 
\end{equation}
for any $\delta>0$.

Let $Q_i,Q_{i+1}$ be two consecutive cubes in the chain $C(Q,w)$.
By construction, these are two intersecting Whitney cubes. Let us
first assume that $\ell(Q_i)\ge \ell(Q_{i+1})$. Then both cubes
are contained in $3Q_i$
and their measures  are comparable with a constant depending on $n$ only.
Note also that by~\eqref{dist_est} we have $\mathrm{int}(3Q_i)\sub G$.
By the fractional $(1,q,\beta)$-Poincar\'e inequality for cubes, Lemma \ref{lemma:poincare_cube}, \begin{equation}\label{eq:two cubes}\begin{split}
  |u_{Q_i}-u_{Q_{i+1}}| & \leq |u_{Q_i}-u_{3Q_i}| + |u_{3Q_i}-u_{Q_{i+1}}|\\
    & \le C |3Q_i|^{-1}\int_{3Q_i}|u(x)-u_{3Q_i}|\,dx\\
    & \le C |3Q_i|^{\beta/n - 1/q} \biggl(\int_{3Q_i}\int_{3Q_i}
          \frac{|u(x)-u(y)|^q}{|x-y|^{n+\beta q}}\,dy\,dx\biggr)^{1/q}.
 \end{split}
 \end{equation}
 In case of $\ell(Q_i)< \ell(Q_{i+1})$, we have the same estimate but with $3Q_{i+1}$ on the last line.
 Since $u(w)=0$ and $|u_Q|=1$, it follows from~\eqref{eq:basic} that
 \begin{equation}\label{eq:chain}
 \begin{split}
  1 & = |u(w)-u_{Q}| \le \sum_{i=0}^\infty |u_{Q_i}-u_{Q_{i+1}}|\\
    & \le C \sum_{i=0}^\infty \ell(Q_i)^{\beta - n/q}\biggl(\int_{3Q_i}\int_{3Q_i}
          \frac{|u(x)-u(y)|^q}{|x-y|^{n+\beta q}}\,dy\,dx\biggr)^{1/q}.
 \end{split}
 \end{equation}

Comparison of the sums in \eqref{eq:delta sum} and \eqref{eq:chain} leads to the observation
 that, for each $\delta>0$, there exists  an index $i_w\ge 0$ such that
 \[
 \biggl(\frac{\ell(Q_{i_w})}{\ell(Q)}\biggr)^{\delta}\le C \ell(Q_{i_w})^{\beta-n/q}\biggl(\int_{3Q_{i_w}}\int_{3Q_{i_w}}
          \frac{|u(x)-u(y)|^q}{|x-y|^{n+\beta q}}\,dy\,dx\biggr)^{1/q}\,.
 \]
 We now choose $\delta=(\lambda-n+\beta q)/q>0$,
 and so, writing $Q_{i_w}=Q_w$, we have found 
 a cube $Q_w\in C(Q,w)$ satisfying
 \begin{equation}\label{eq:good cube}
 \ell(Q_w)^{\lambda}\leq C \ell(Q)^{\lambda-n+\beta q}
    \int_{3Q_w}\int_{3Q_w}\frac{|u(x)-u(y)|^q}{|x-y|^{n+\beta q}}\,dy\,dx\,.
 \end{equation}

Next we claim that $3Q_w\subset LQ$,
where $L=L(c,n)=30 c \sqrt n\ge 1$.
Indeed, if $y\in 3Q_w$, then inequalities \eqref{eq:useful} and the arc-length parametrization
of $\gamma$ imply that
\begin{align*}
\lvert y-x_Q\rvert &\le \lvert y - \gamma(t_{i_w}) \rvert + \lvert \gamma(t_{i_w})-\gamma(t_0)\rvert
\\&\le \diam(3Q_{w}) + t_0-t_{i_w} \le 3t_0 \le 15 c \diam(Q)=30 c \sqrt n \ell(Q)/2
\end{align*}
and hence $3Q_w\subset LQ$.
Analogous estimates also show that
\begin{equation}\label{e.o2}
w\in B_w:=B(x_{Q_w},\rho\, \ell(Q_w))\quad \text{ and } \quad 3Q_{\omega}\subset B_\omega\,,
\end{equation}
where $\rho=\rho(c,n)=11 c \sqrt n>0$.

 We can now finish the proof.
 Since a cube as in \eqref{eq:good cube} exists for each $w\in \vbdry_{x_Q,c} G$, we may use
 the $5r$-covering lemma and compactness to cover the set $\vbdry_{x_Q,c} G$ with balls $5B_{w_j}$,
 centered in $w_i\in \vbdry_{x_Q,c} G$, $j=1,\dots,N$,
 in such a way that the balls $B_j=B_{w_j}$ are pairwise disjoint and
 $\vbdry_{x_Q,c}G \sub \bigcup_{j=1}^N 5 B_j$.

 Let us write $Q^j=Q_{w_j}$.
 Now estimate \eqref{eq:good cube} and the pairwise disjointness (since $3Q^j\sub B_j$) of the
 cubes $3Q^j$, satisfying $\mathrm{int}(3Q^j)\sub LQ\cap G$, lead to
  \begin{equation}\label{eq:Hastimate}
   \begin{split}
  \Ha^\lambda_{\infty}(\vbdry_{x_Q,c}G)
  & \leq \sum_{j=1}^N (5\rho \ell(Q^j))^{\lambda} \leq C\sum_{j=1}^N \ell(Q^{j})^{\lambda} \\
  & \leq C\sum_{j=1}^N \ell(Q)^{\lambda-n+\beta q}\int_{3Q^j}\int_{3Q^j}\frac{|u(x)-u(y)|^q}{|x-y|^{n+\beta q}}\,dy\,dx\\
  & \leq C\sum_{j=1}^N \ell(Q)^{\lambda-n+\beta q}\int_{3Q^j}\int_{LQ\cap G}
              \frac{|u(x)-u(y)|^q}{|x-y|^{n+\beta q}}\,dy\,dx\\
  & \leq C \ell(Q)^{\lambda-n+\beta q}\int_{LQ\cap G}\int_{LQ\cap G}
              \frac{|u(x)-u(y)|^q}{|x-y|^{n+\beta q}}\,dy\,dx.
   \end{split}
  \end{equation}
 As we assumed $|u_Q|=1$, the claim follows
 from \eqref{eq:Hastimate}.
\end{proof}

The next lemma, essentially~\cite[Lemma~3.4]{ihnatsyeva1}, is also needed in the
proof of Theorem~\ref{thm:fractional_hardy_visible}. For the sake
of convenience, we include also a proof. Notice that inequality~\eqref{beta_ineq} is not a 
direct consequence of H\"older's inequality, since the cubes
$L Q$ need not have bounded overlap if $L$ is large; this is indeed the case in our applications.

\begin{lemma}\label{carleson_lemma}
Assume that $G\subset\R^n$ is an open set, $n\ge 2$.
If $1<\kappa <\infty$ and $L \ge 1$, then, for every $g\in L^\kappa(\R^n\times \R^n)$,
\begin{equation}\label{beta_ineq}
 \sum_{Q\in\W(G)} |Q|^{2}\bigg(\vint_{L Q}\vint_{L Q} |g(x,y)|\,dy\,dx\bigg)^\kappa
 \le C \int_{\R^n}\int_{\R^n}  |g(x,y)|^\kappa\,dy\,dx\,,
\end{equation}
where the constant $C>0$ only depends on $n$ and $\kappa$.
\end{lemma}

\begin{proof}
Throughout this proof, we write $m=2n$.
Recall that the (non-centered) maximal function (with respect to cubes) of a locally integrable function
$f\colon \R^{m}\to [-\infty,\infty]$ is
\[
\mathcal{M} f(x) = \sup_{x\in P} \vint_{P} |f(y)|\,dy\,,
\]
where the supremum is taken over all cubes $P$ in
$\mathbb{R}^m$ containing $x\in \mathbb{R}^m$.

Rewriting the left-hand side of inequality \eqref{beta_ineq} and using the
definition of the maximal function in the point $(z,w)\in LQ\times LQ=P\sub\R^m$,
we obtain
\begin{align*}
\sum_{Q\in\W(G)} & |Q|^{2}\bigg(\vint_{L Q}\vint_{L Q} |g(x,y)|\,dy\,dx\bigg)^\kappa\\
& = \sum_{Q\in\W(G)} \int_{\R^n}\int_{\R^n} \chi_{Q}(z)\chi_{Q}(w)
     \bigg(\vint_{L Q} \vint_{L Q}|g(x,y)|\,dy\,dx\bigg)^\kappa \,dw\,dz\\
& \le \sum_{Q\in\W(G)}\int_{\R^n}\int_{\R^n} \chi_{Q}(z)\chi_Q(w) \big(\mathcal{M} g(z,w)\big)^\kappa\,dw\,dz\\
& \le \int_{G}\int_{G} \big(\mathcal{M} g(z,w)\big)^\kappa\,dw\,dz\,.
\end{align*}
The boundedness of the maximal operator $\mathcal{M}$
on $L^\kappa(\R^m)$ yields the claim.
\end{proof}

We are now ready to prove our main result:

\begin{proof}[Proof of Theorem \ref{thm:fractional_hardy_visible}]
Fix a number $q\in [1,p)$ such that $\beta=n(1/p-1/q)+s \in (0,s )$
and
\begin{equation}\label{mu2_def}
 n-\lambda<\beta q <sp,\qquad 0<1/q-1/p<\beta/n\,.
\end{equation}
Let $Q\in\W(G)$, and let
$L\geq 1$ be as in Lemma \ref{lemma:Hausdorff estimate}.
We first claim that
\begin{equation}\label{upper_bound}
\begin{split}
\int_{Q} \lvert u(x) & \rvert^p\,dx  \leq C \bigg[\int_{Q} \lvert u(x)-u_{{Q}}\rvert^p\,dx +
\lvert Q\rvert \lvert u_{{Q}}\rvert^p\bigg]\\
&\le C |Q|^{1+p\beta /n-p/q} \bigg(\int_{LQ\cap G}\int_{LQ\cap G}
\frac{|u(x)-u(y)|^q}{|x-y|^{n+\beta q}}\,dy\,dx\bigg)^{p/q}.
\end{split}
\end{equation}
For the term $\int_{Q} \lvert u(x)-u_{{Q}}\rvert^p\,dx$, the second inequality in
\eqref{upper_bound} is a consequence of the $(p,q,\beta)$-Poincar\'e inequality of
Lemma~\ref{lemma:poincare_cube}.
On the other hand, for
$\lvert Q\rvert \lvert u_{{Q}}\rvert^p$
the inequality follows from Lemma~\ref{lemma:Hausdorff estimate}.
Indeed, by the assumption~\eqref{eq:main cond} on the visual boundary
and inequality \eqref{eq:bound on Ha}, we have
\begin{equation*}
\begin{split}
  \ell(Q)^\lambda|u_Q|^{q} & \leq C \Ha_{\infty}^\lambda (\bdry^{\vis}_{x_Q,c}G)|u_Q|^{q} \\
     & \leq C  \ell(Q)^{\lambda - n + \beta q} \int_{LQ\cap G}\int_{LQ\cap G}
         \frac{|u(x)-u(y)|^q}{|x-y|^{n+\beta q}}\,dy\,dx\,,
\end{split}
\end{equation*}
and so
\begin{equation*}
\begin{split}
  |Q||u_Q|^{p} \leq C  |Q|^{1+p\beta/n-p/q} \bigg(\int_{LQ\cap G}\int_{LQ\cap G}
         \frac{|u(x)-u(y)|^q}{|x-y|^{n+\beta q}}\,dy\,dx\bigg)^{p/q},
\end{split}
\end{equation*}
proving~\eqref{upper_bound}.

Now property \eqref{dist_est} of Whitney cubes and inequality~\eqref{upper_bound} imply
\begin{equation}\label{eq:split to cubes}
\begin{split}
& \int_{G} 
\frac{\vert u(x)\rvert^p} {\dist (x,\partial G)^{sp}} \,dx
\le \sum _{Q\in \W(G)} \diam (Q)^{-sp}\int_{Q}\vert u(x)\vert^p\,dx
\\&\le C \sum _{Q\in \W(G)}
|Q|^{-sp/n}|Q|^{1+ p\beta/n-p/q} \bigg(|Q|^2\vint_{LQ}\vint_{LQ}
g(x,y)\,dy\,dx\bigg)^{p/q}
\\&\le C \sum _{Q\in \W(G)}
|Q|^{1+ p\beta/n+p/q-sp/n} \bigg(\vint_{LQ}\vint_{LQ}
g(x,y)\,dy\,dx\bigg)^{p/q}\,,
\end{split}
\end{equation}
where we have written
\[
 g(x,y) = \chi_G(x)\chi_G(y)\, \frac{|u(x)-u(y)|^q}{|x-y|^{n+\beta q}  }\,.
\]
By the choice of $\beta$ we have in the last line of~\eqref{eq:split to cubes} that
\begin{equation*}
1+ p\beta /n+p/q-sp/n= 2\,,
\end{equation*}
and thus Lemma~\ref{carleson_lemma}, applied to the above function $g$ with $\kappa=p/q>1$, yields
together with estimate~\eqref{eq:split to cubes} that
\begin{equation*}
\begin{split}
\int_{G} \frac{\vert u(x)\rvert^p} {\dist (x,\partial G)^{sp}} \,dx
& \le C \int_{\R^n}\int_{\R^n} g(x,y)^{p/q} \,dy\,dx\\
& \le C \int_{G}\int_{G}
\frac{\vert u(x)-u(y)\vert^p}{\vert x-y\vert ^{n+ sp}}
\,dx\,dy\,.
\end{split}
\end{equation*}
Note that we used above also the identity $p(n+\beta q)/q = n+sp$
to obtain the correct exponent in the denominator.
\end{proof}

The following corollary is a consequence of Theorem~\ref{thm:fractional_hardy_visible}
and Proposition~\ref{prop:unif}.

\begin{corollary}\label{cor:fractional_hardy_uniform_thick}
Let $0<s<1$ and $1<p<\infty$ satisfy $0<sp<n$. Assume that $G\sub\R^n$ is a uniform domain 
and that there exists $n - sp<\lambda \le n $ and $C>0$ such that
\begin{equation}\label{e.inner}
\mathcal{H}^\lambda_\infty(B(x,2\dist(x,\partial G))\cap \partial G) \ge C\dist(x,\partial G)^\lambda
\end{equation}
for all $x\in G$. 
Then $G$ admits an $(s,p)$-Hardy inequality.
\end{corollary}

Observe that Corollary~\ref{cor:fractional_hardy_uniform}, stated in the Introduction, 
is now  a consequence of Corollary~\ref{cor:fractional_hardy_uniform_thick}
and Proposition~\ref{prop:fatness implies density}. We also
refer to the discussion in the beginning of Section~\ref{s.comparison}.

We conclude this section with a partial relaxation of the
visual boundary condition of Theorem~\ref{thm:fractional_hardy_visible}.
The following modification of Lemma~\ref{lemma:Hausdorff estimate}
shows that we do not need to assume the visibility of the boundary
with respect to cubes $Q\in\W(G)$ if the nearby complement is of zero measure. 

\begin{lemma}\label{lemma:Hausdorff estimate no acces}
Let $G\sub\R^n$ be an open set and assume that
$0<\beta<1$, $1\le q<\infty$, and $n-\beta q < \lambda \le n$.
Fix $c\geq 1$ and let $Q\in\W(G)$.
Assume that $L>1$ and $|G^c\cap LQ|=0$.
Then there exists $C=C(n,c,\lambda,q,\beta,L)>0$, independent of $Q$, such that
\begin{equation}\label{eq:bound on Ha_II}
 \Ha_{\infty}^\lambda (\bdry G \cap LQ)|u_Q|^{q} \leq
     C  \ell(Q)^{\lambda-n+\beta q} \int_{LQ\cap G}\int_{LQ\cap G}
        \frac{|u(x)-u(y)|^q}{|x-y|^{n+\beta q}}\,dy\,dx
\end{equation}
for every $u\in C_0^\infty(G)$.
\end{lemma}

\begin{proof}
 Let us indicate here the main differences compared to the proof of Lemma~\ref{lemma:Hausdorff estimate}. 
 For $w\in \bdry G \cap LQ$ we can now consider, instead of the John-type chain $C(Q,w)$, a chain
 $\widetilde C(Q,w)=\{Q_0,Q_1,\ldots\}$ of cubes, where $Q_0=Q$, $Q_1=LQ$, $w\in Q_{i+1}\sub Q_{i}$ for every $i\ge 1$,
 and $\ell(Q_{i+1})=\ell(Q_{i})/2$ for every $i\ge 1$. Estimate \eqref{eq:chain} (but with the integrals taken over $Q_i$)
 is a simple consequence of the $(1,q,\beta)$-Poincar\'e inequality for cubes, 
 and thus we find from the chain $\widetilde C(Q,w)$
 a cube $Q_w$ as in \eqref{eq:good cube} (note that here the constant depends on $L$).
 We can now choose disjoint cubes $Q_j=Q_{w_j}$ 
 such that the dilated cubes $5Q_j$ cover the
 set $\bdry G \cap LQ$, and thus we obtain, just like in the proof of Lemma \ref{lemma:Hausdorff estimate}, that
   \begin{equation}\label{eq:Hastimate2}
   \begin{split}
  \Ha^\lambda_{\infty}(\bdry G \cap LQ)
  & \leq C\sum_{j=1}^N \ell(Q_{j})^{\lambda}\\
  & \leq C\sum_{j=1}^N \ell(Q)^{\lambda-n+\beta q}\int_{Q_j}\int_{Q_j}\frac{|u(x)-u(y)|^q}{|x-y|^{n+\beta q}}\,dy\,dx\\
  & \leq C\sum_{j=1}^N \ell(Q)^{\lambda-n+\beta q}\int_{Q_j}\int_{LQ}
              \frac{|u(x)-u(y)|^q}{|x-y|^{n+\beta q}}\,dy\,dx\\
  & \leq C \ell(Q)^{\lambda-n+\beta q}\int_{LQ}\int_{LQ}
              \frac{|u(x)-u(y)|^q}{|x-y|^{n+\beta q}}\,dy\,dx\\
  & \leq C \ell(Q)^{\lambda-n+\beta q}\int_{LQ\cap G}\int_{LQ\cap G}
              \frac{|u(x)-u(y)|^q}{|x-y|^{n+\beta q}}\,dy\,dx.
   \end{split}
  \end{equation}
  The last line above follows from the assumption $|G^c\cap LQ|=0$.
\end{proof}

\begin{remark}\label{rmk:zero complement}
Let $G\sub\R^n$ be an open set and assume that there exists $0<\lambda\le n$, $c\ge 1$, $L\ge 1$, and $C>0$
such that for every $Q\in\W(G)$ (at least) one of the following conditions holds:
\begin{itemize}
\item[(a)] $\Ha^\lambda_\infty\big(\bdry^{\vis}_{x_Q,c} G\big)\geq C \ell(Q)^\lambda$\,; or,

\item[(b)] $|G^c\cap LQ|=0$ and $\Ha^\lambda_\infty\big(\bdry G \cap LQ \big)\geq C \ell(Q)^\lambda$.
\end{itemize}
Then we can use either Lemma~\ref{lemma:Hausdorff estimate} or
Lemma~\ref{lemma:Hausdorff estimate no acces} to yield estimate~\eqref{upper_bound}
in the proof of Theorem~\ref{thm:fractional_hardy_visible}, and we conclude that
such an open set $G$ admits $(s,p)$-Hardy inequalities for all $1<p<\infty$ and $0<s<1$ satisfying
$n-\lambda<sp<n$. 

Note that in the case $|G^c|=0$, where we can use the above case (b) for every cube $Q\in\W(G)$, 
the same conclusion also follows from Theorem~\ref{t.equiv} and Theorem~\ref{fractional_hardy_fat} below.
\end{remark}

\section{Counterexamples in the plane}\label{s.counterexamples}

In this section, we show that $(s,p)$-uniform fatness of the boundary
$\bdry G$ is not sufficient for $G$ to admit an $(s,p)$-Hardy inequality. Similar
examples exist in higher dimensions as well, but, for the sake of clarity, we confine 
ourselves here to the planar case. Let us formulate this as a theorem:

\begin{theorem}\label{t.counter}
Let $1<p<\infty$ and $0<s<1$ be such that $0<sp\le 1$. Then there exists
an open set $G\sub\R^2$ whose boundary is $(s,p)$-uniformly fat, but which does not admit
$(s,p)$-Hardy inequalities. In addition, there exists a bounded domain $G'\sub\R^2$,
with a locally $(s,p)$-uniformly fat boundary, such that $G'$ does not admit
$(s,p)$-Hardy inequalities.
\end{theorem}

We base our constructions on the examples
in~\cite{Dyda}, but we need significant modifications
in order to adapt them for our purposes.
To clarify this, let  $1<p<\infty$ and $0<s<1$ be such that $sp\le 1$.
Then the domain \[{G}^{\textup{core}}:= (-1,1)^2\subset \R^2\,\]
does not admit an $(s,p)$-Hardy inequality,
we refer to \cite{Dyda} or the proof of Proposition \ref{l.suffices}.
Albeit the complement of $G^{\textup{core}}$ is $(s,p)$-uniformly fat,
this counterexample does not suffice for
the proof of Theorem~\ref{t.counter}, as the boundary of ${G}^{\textup{core}}$ is still too `thin';
it is $(s,p)$-uniformly fat only for $sp>1$, exactly the same range for which the set ${G}^{\textup{core}}$
admits $(s,p)$-Hardy inequalities.

We give two possible ways to address this problem: 
In the first one, given in Sections~\ref{s.outline}--\ref{s.G-construction},
we place a fat Cantor set---one with positive measure---to each of the Whitney cubes
in $\mathcal{W}(\R^2\setminus \bar{G}^{\textup{core}})$.
This is done in a quantitative manner, ensuring that $(s,p)$-Hardy inequalities still remain false.
The construction yields an open set $G$ whose boundary is an unbounded $2$-regular set, and
consequently, by Proposition~\ref{t.lambda_sets_are_fat}, the boundary $\partial G$ is indeed $(s,p)$-uniformly fat,
thus proving the first claim in Theorem~\ref{t.counter}.
The second claim is proved in Sections~\ref{s.local halfspace}--\ref{s.local sf}. There we consider localized examples,
and the resulting sets $G$ are even domains. However, in these latter examples much more care must be taken in the choices 
of the test functions and in the related calculations.

\subsection{Outline of the first example}\label{s.outline}
The construction consists of two steps. 
In the {\em first step} we construct a John domain $G^{\textup{core}}$ satisfying the following conditions (A)--(D):
\begin{itemize}
\item[(A)]
 there is a constant $C>0$ such that for any $m\in \N$, the boundary $\partial G^{\textup{core}}$ can be covered
by using at most $Cm^{2-sp}$ balls in the family
$\{B(x,1/(2m))\,:\,x\in \partial G^{\textup{core}}\}$;
\item[(B)]
$\int_{G^{\textup{core}}}\dist(x,\partial G^{\textup{core}})^{-sp}\,dx = \infty$;
\item[(C)]
$\partial G^{\textup{core}}=\partial(\R^2\setminus \bar G^\textup{core})$
and $\lvert \partial G^{\textup{core}}\rvert=0$;
\item[(D)]
the following `co-plumpness' condition is satisfied
for some $\eta\in (0,1)$: 
for each $x\in \R^2\setminus G^{\textup{core}}$ and each $r> 2\dist(x,\partial G^{\textup{core}})$, there is
$Q\in \mathcal{W}(\R^2\setminus \bar G^{\textup{core}})$
such that $\lvert Q\rvert \ge \eta r^2$ and $Q\subset B(x,r)$.
\end{itemize}
Conditions (A) and (B) are taken from~\cite{Dyda}, and they imply that $G^{\textup{core}}$ does not
admit an $(s,p)$-Hardy inequality. The two other conditions (C) and (D)
are technical, required in the second step of the construction.
In the case of $sp=1$, we may begin our construction with $G^{\textup{core}}:=(-1,1)^2$.
For the remaining  cases $0<sp<1$, we refer to Section~\ref{s.sel}.

In the {\em second step}, 
we construct an open set $G\sub\R^2$ with
the following properties:
\begin{itemize}
\item[(1)] $G^{\textup{core}}\subset G$;
\item[(2)] $\partial G^{\textup{core}} \subset \partial G$;
\item[(3)] $\sum_{j=-\infty}^\infty 2^{jsp}\sum_{Q\in\mathcal{W}_j(\R^2\setminus \bar G^{\textup{core}})}\lvert Q\cap G\rvert <\infty$;
\item[(4)] $\partial G$ is $(s,p)$-uniformly fat.
\end{itemize}
The construction of such sets is given in Section~\ref{s.G-construction}. 
Let us now show 
that a set $G$ satisfying (1)--(4) indeed gives a desired counterexample:
\begin{proposition}\label{l.suffices}
Suppose that an open set $G$ satisfies
conditions (1)--(3) above. Then $G$ does not admit an $(s,p)$-Hardy inequality.
\end{proposition}

\begin{proof}
We use test functions $u_m\in C^\infty_0(G^\textup{core})$
with $m\ge 1$ large enough. These are defined by
\[
u_m := \varphi_{8m} * \chi_{\{x\in G^{\textup{core}}\,:\, \dist(x,\partial G^{\textup{core}})\ge 3/(8m)\}}\,,
\]
where $\varphi_m = m^2 \varphi(mx)$ and $\varphi\in C^\infty_0(B(0,1))$
is a non-negative bump function, with
$\int \varphi \,dx=1$.
 Some of the computations in \cite{Dyda} are invoked in the sequel, and actually the only
properties we explicitly need for the functions $u_m$ are that $0\le u_m\le 1$ for all $m$,
and that they converge pointwise to $\chi_{G^{\textup{core}}}$ monotonically,
i.e. $u_{m+1}\ge u_m$ for all $m$.

By conditions (1), (2), and (B),
\[
\int_G \frac{\lvert u_m(x)\rvert^p}{\dist(x,\partial G)^{sp}}\,dx
\ge \int_{G^{\textup{core}}} \frac{\lvert u_m(x)\rvert^p}{\dist(x,\partial G^{\textup{core}})^{sp}}\,dx
\xrightarrow{m\to\infty}\,\infty\,.
\]
Hence, it suffices to show that the following integrals are uniformly bounded in $m$:
\begin{align*}
&\int_G \int_G \frac{\lvert u_m(x)-u_m(y)\rvert^p}{\lvert x-y\rvert^{2+sp}} \,dy\,dx\\
&= \bigg\{\int_{G^{\textup{core}}}\int_{G^{\textup{core}}}
+ 2\int_{G\setminus G^{\textup{core}}}\int_{G^{\textup{core}}}\bigg\}
\frac{\lvert u_m(x)-u_m(y)\rvert^p}{\lvert x-y\rvert^{2+sp}} \,dy\,dx\,.
\end{align*}
The integral
over the set $G^{\textup{core}}\times G^{\textup{core}}$ is uniformly bounded
in $m$; we refer to \cite[pp.\ 577--578]{Dyda} for a computation which relies on
the covering property (A).
Hence, it suffices to estimate the second integral,
\begin{align*}
S_m:&=\int_{G\setminus G^{\textup{core}}}\int_{G^{\textup{core}}}
\frac{\lvert u_m(x)-u_m(y)\rvert^p}{\lvert x-y\rvert^{2+sp}} \,dy\,dx\\
&= \sum_{j=-\infty}^\infty \sum_{Q\in\mathcal{W}_j(\R^2\setminus \bar G^{\textup{core}})}
\int_{G\cap Q}\int_{G^{\textup{core}}}\frac{\lvert u_m(y)\rvert^p}{\lvert x-y\rvert^{2+sp}} \,dy\,dx\,;
\end{align*}
we used above condition~(C), which assures that  $\partial G^{\textup{core}}$ has zero measure.
Again, by condition~(C), for every $x\in G\cap Q\subset \R^2\setminus \bar G^{\textup{core}}$
(with $Q$ as in the summation above)
and $y\in G^{\textup{core}}$,
\[
\lvert x-y\rvert \ge \dist(x,\partial G^{\textup{core}})
=\dist(x,\partial (\R^2\setminus \bar G^{\textup{core}}))\ge \diam(Q) \ge  2^{-j}\,.
\]
Since the test functions  $u_m$ are bounded by $1$, we may integrate
in polar coordinates in order to see that
\begin{align*}
S_m &\le \sum_{j=-\infty}^\infty \sum_{Q\in\mathcal{W}_j(\R^2\setminus \bar G^{\textup{core}})}
  \int_{G\cap Q}\int_{\R^2\setminus B(x,2^{-j})} \frac{1}{\lvert x-y\rvert^{2+sp}}\,dy\,dx\\
  &\lesssim \sum_{j=-\infty}^\infty 2^{jsp}\sum_{Q\in\mathcal{W}_j(\R^2\setminus \bar G^{\textup{core}})} \lvert G\cap Q\rvert<\infty\,.
\end{align*}
The last step follows from condition (3).
\end{proof}

\subsection{Construction of $G^{\textup{core}}$}\label{s.sel}
Let $1<p<\infty$ and $0<s<1$ be such that $sp<1$.
We construct a domain satisfying properties (A)--(D)
from the beginning of Section~\ref{s.outline}.
Let $\lambda=2-sp$ and let $G^{\textup{core}}$ be a snowflake domain,
which is a co-plump John domain
whose boundary $E:=\partial G^{\textup{core}}$
is a $\lambda$-regular set and satisfies $E=\partial (\R^2\setminus \bar G^{\textup{core}})$.
The usual von Koch snowflake, illustrated in Figure~1, 
has dimension $\lambda=\log 4/\log 3$, 
but the construction
can be modified to obtain a set $E$ of any dimension $\lambda\in (1,2)$;
see for instance~\cite[Sect.~2]{KOSKELA}.
In particular, $E$ has zero Lebesgue measure. Thus, conditions (C) and (D)
are satisfied by the construction

\begin{figure}[!htb]
\begin{center}
\includegraphics[width=6cm]{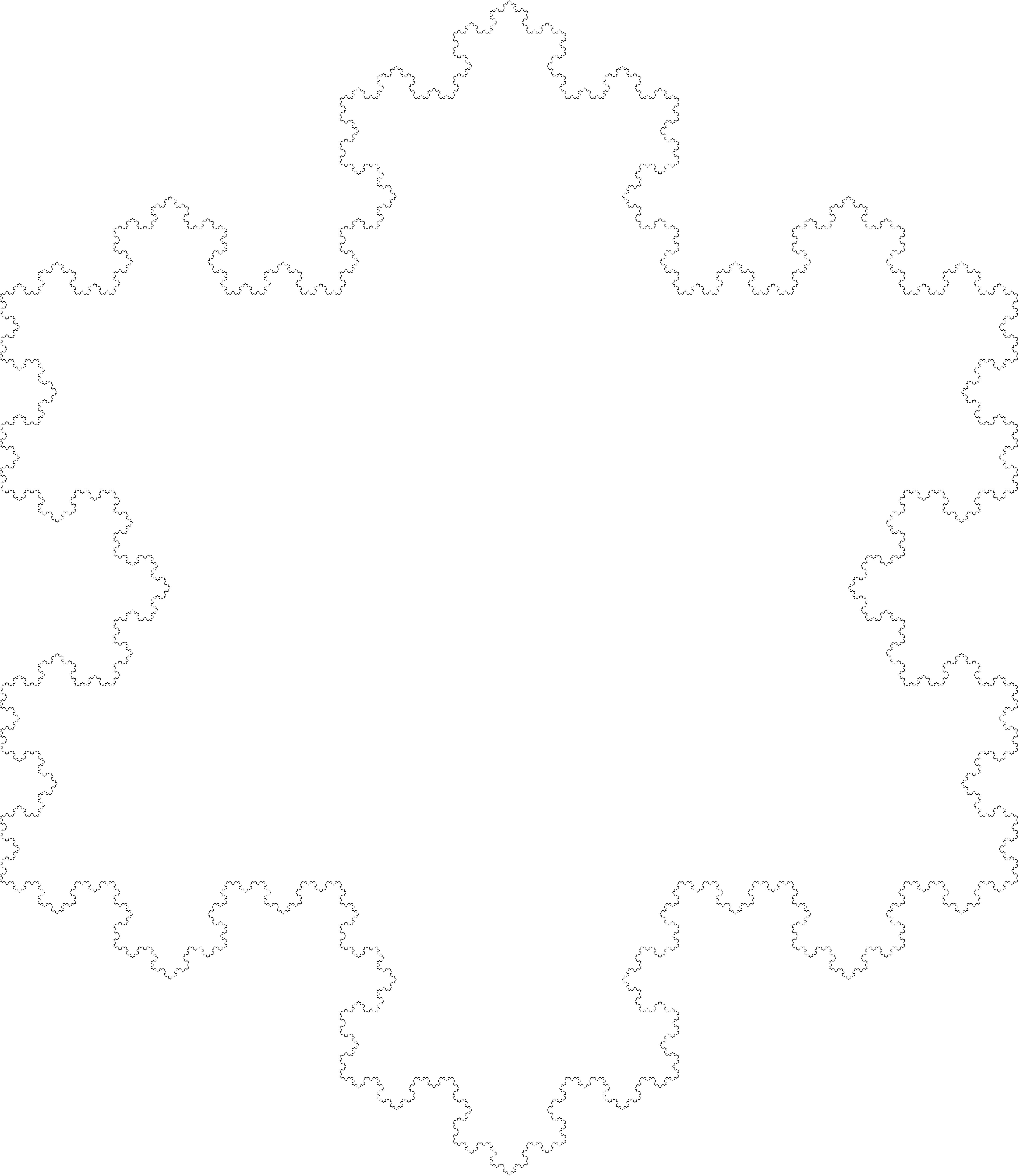}
\end{center}
\caption{ The usual von Koch snowflake. }
\end{figure}

In order to verify the remaining two conditions (A) and (B), we
use ideas from~\cite[Lemma 2.1]{lehrback}. Fix $0<r\le 1/5$ and
consider the family $\mathcal{F}_r=\{B(x,r)\,:\,x\in E\}$. By the $5r$-covering lemma, 
there are pairwise disjoint balls
$B_1^r,\ldots,B_{N_r}^r\in \mathcal{F}_r$ such that
$E\subset\cup_{i=1}^{N_r}5B_i^r$. By the $\lambda$-regularity condition
\eqref{e.lambda_set} and the pairwise disjointedness of the balls $B_i^r$, we
have
\begin{align*}
\mathcal{H}^\lambda (E) & \le \sum_{i=1}^{N_r}\mathcal{H}^\lambda (5B_i^r\cap E)
\le C5^\lambda  N_rr^\lambda \\ 
& \le C^25^\lambda\sum_{i=1}^{N_r} \mathcal{H}^\lambda (B_i^r\cap E )
\le C^2 5^\lambda\mathcal{H}^\lambda(E),
\end{align*}
and, hence, 
$C_1 r^{-\lambda} \le N_r \le C_2r^{-\lambda}$, where
the positive constants $C_1$ and $C_2$ do not depend on $r$, in particular
$N_r\simeq r^{-\lambda}$.
As a consequence, we obtain the covering property (A) for $G^{\textup{core}}$.

It remains to verify condition (B).
By denoting
\[E_r=\{x\in G^{\textup{core}}\,:\,\mathrm{dist}(x,E)\le r\}\,,\]
we have
\begin{align*}
&\int_{G^{\textup{core}}} \dist(x,E)^{\lambda-2}\,dx\\
&= \int_0^\infty |\{x\in G^{\textup{core}}\,:\,\dist(x,E)^{\lambda-2}\ge  t\}|\,dt
= \int_{0}^\infty |E_{t^{1/(\lambda-2)}}|\,dt\,.
\end{align*}
Thus, for the last condition, it suffices to show that, for some $C>0$ independent of $r$,
\begin{equation}\label{suff}
|E_r|\ge Cr^{2-\lambda},\quad r\in (0,1/5).
\end{equation}
We can deduce this estimate as follows. Let us first
observe that, by John property of $G^{\textup{core}}$, there
is a constant $C$ such that $|B_i^r \cap G^{\textup{core}}|\ge C|B_i^r|$. Hence,
by the pairwise disjointedness of the balls $B_i^r$,
\begin{align*}
|E_r| &\ge \bigg|\bigcup_{i=1}^{N_r} B_i^r\cap G^{\textup{core}}\bigg|
=\sum_{i=1}^{N_r} |B_i^r \cap G^{\textup{core}}|
\\&\ge C\sum_{i=1}^{N_r} |B_i^r|
=C\sum_{i=1}^{N_r} r^2 = CN_r r^2
\ge Cr^{2-\lambda},
\end{align*}
and we conclude that the
property (B) holds, i.e. $\int_{G^{\textup{core}}} \mathrm{dist}(x,E)^{-sp}\,dx = \infty$.

\subsection{Fat Cantor sets}\label{s.aux_compact}
We construct next an auxiliary compact set $K_j=K_j(Q_j)\subset Q_j$.
Here $Q_j\subset \R^2$ is a given closed cube with
$\ell(Q_j)=2^{-j}$, $j\in \Z$, and the construction is
parameterized by a given sequence $\{\varepsilon_{j,k}\}_{k\ge 0}$ of
real numbers in $(0,1/2)$ such that $\sum_{k=0}^\infty
\varepsilon_{j,k} < 1/2$.

\begin{figure}[!htb]
\begin{center}
\includegraphics[width=6cm]{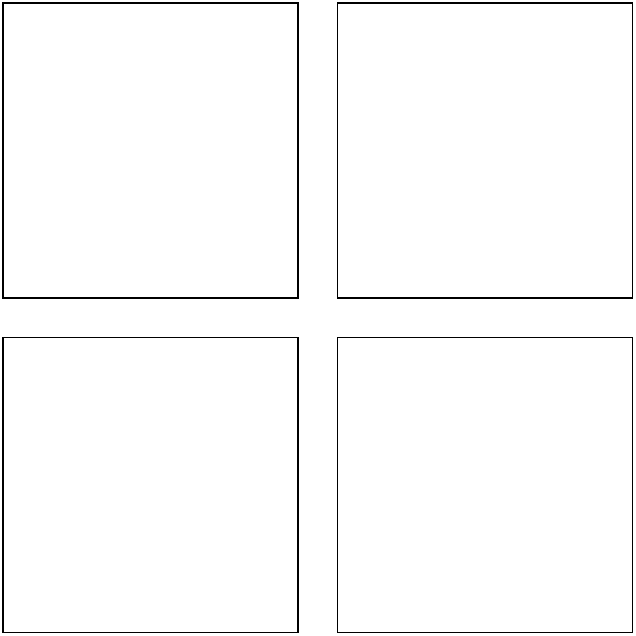}
\end{center}
\caption{ We have removed two rectangles from a closed cube
$Q_{j,k,m}$. The four closed sub cubes belong to generation $k+1$,
and they are denoted by $Q_{j,k+1,m_1},\ldots,Q_{j,k+1,m_4}$ for
some $\{m_1,\ldots,m_4\}\in \{1,2,\ldots, 4^{k+1}\}$. }
\end{figure}

We construct cubes $Q_{j,k,m}$ with common side length $\ell_{j,k}$, where
$m\in \{1,\ldots 4^k\}$ is used to index the cubes of generation $k\ge 0$. 
Let $Q_{j,0,1}=Q_{j}$ and define $\ell_{j,0}=2^{-j}$. At the inductive stage, we are given $k\ge 0$ and the cubes
$Q_{j,k,m}$
whose side lengths are $\ell(Q_{j,k,m})=\ell_{j,k}$, $m=1,\ldots, 4^k$. Each of these cubes is partitioned in five sets:
to the union of two overlapping perpendicular rectangles, both similar to
 $[0,\ell_{j,k}]\times (0,\varepsilon_{j,k}\ell_{j,k})$ and having midpoint $x_{Q_{j,k,m}}$; and to four similar
closed cubes of side length $\ell_{j,k+1}=(1-\varepsilon_{j,k})\ell_{j,k}/2$.
The resulting $4^{k+1}$ `generation $k+1$' cubes are denoted by $Q_{j,k+1,m}$, $m=1,\ldots,4^{k+1}$; i.e.,
we fix an ordering for these cubes. 
For an illustration, we refer to Figure~2.
It is immediate from the construction that
\begin{equation}\label{e.size_estimate}
\ell_{j,k}/4 \le \ell_{j,k+1}\le \ell_{j,k}/2\text{ and } \ell_{j,k'}\le 2^{k-k'}\ell_{j,k}\text{ if } k'\ge k\,.
\end{equation}

We define compact sets $K_{j,0}\supset K_{j,1}\supset\dotsb\supset K_j$ by setting
\[
K_{j,k} := \bigcup_{m=1}^{4^k} Q_{j,k,m}\,,\qquad K_j:=\bigcap_{k\ge 0}K_{j,k}\,.
\]
The following `quantitative  Lebesgue density theorem' for $K_j$ will be useful
in several occasions.
\begin{lemma}\label{l.diff}
Suppose that $k\ge 0$ and $m\in \{1,2,\ldots, 4^k\}$ are fixed. Then
\begin{equation}\label{e.recursive}
\lvert K_{j} \cap Q_{j,k,m}\rvert\ge \lvert Q_{j,k,m}\rvert\bigg(1 -  \sum_{k'=k}^{\infty} 2\varepsilon_{j,k'} \bigg)\,.
\end{equation}
\end{lemma}

\begin{proof}
Since
$\lvert K_{j} \cap Q_{j,k,m}\rvert = \lim_{k'\to\infty} \vert K_{j,k'} \cap Q_{j,k,m}\vert$,
it suffices to prove the inequality
\begin{equation}\label{e.induction}
\vert K_{j,k'} \cap Q_{j,k,m}\vert\ge \vert Q_{j,k,m}\vert\bigg(1 -  \sum_{k''=k}^{k'-1} 2\varepsilon_{j,k''} \bigg)\,,\qquad k'\ge 0\,.
\end{equation}
Let us observe that $K_{j,k'}\cap Q_{j,k,m}\supset K_{j,k}\cap Q_{j,k,m}=Q_{j,k,m}$ for $0\le k'\le k$; hence,
inequality \eqref{e.induction} holds for these values of $k'$.
In the inductive stage, we assume \eqref{e.induction} for some $k'\ge k$. Then, by construction, the intersection
$K_{j,k'}\cap Q_{j,k,m}$ is a union of $4^{k'-k}$ cubes, each of which produces four `generation $k'+1$' cubes
after removal
of a set (which is a union of two perpendicular rectangles)  whose area is strictly bounded by
\[
2\varepsilon_{j,k'}\ell_{j,k'}^2 \le 2\varepsilon_{j,k'} \lvert Q_{j,k,m}\rvert 4^{k-k'}\,.
\]
In the last step above, we applied \eqref{e.size_estimate}.
By the induction assumption, the area of $K_{j,k'+1}\cap Q_{j,k,m}$ is bounded from below by
\begin{align*}
\lvert K_{j,k'}\cap Q_{j,k,m}\rvert - 4^{k'-k}2\varepsilon_{j,k'} \lvert Q_{j,k,m}\rvert4^{k-k'}
&\ge \lvert Q_{j,k,m}\rvert\bigg(1 -  \sum_{k''=k}^{(k'+1)-1} 2\varepsilon_{j,k''}\bigg)\,.
\end{align*}
The proof of inequality \eqref{e.induction} is complete.
\end{proof}

\begin{remark}\label{r.n_set}
A  consequence of Lemma \ref{l.diff} is the following. For every
$x\in K_j$ and every radius $0<r\le 10\sqrt 2\cdot 2^{-j}$,
\[
\lvert B(x,r)\cap K_j\rvert \ge r^2 (10 \sqrt 2)^{-2}\bigg(1-\sum_{k'=0}^\infty 2\varepsilon_{j,k'}\bigg)\,.
\]
(The factor $10\sqrt 2$ is not special, but chosen for our convenience.)
In order to verify this, we choose $k\ge 0$ such that
$\ell_{j,k}/4< r/(10  \sqrt 2) \le \ell_{j,k}$. This is possible by 
inequalities~\eqref{e.size_estimate}. Then we fix an index $m$ in such a way
that $x\in Q_{j,k,m}$. Now $Q_{j,k,m}\subset B(x,r)$, and the
claim follows from inequality~\eqref{e.recursive}  by obvious estimates.
\end{remark}

\subsection{The construction of $G$}\label{s.G-construction}
Finally, we construct an open set $G$ with properties (1)--(4) stated in Section~\ref{s.outline}.
Recall that we have at our disposal a (bounded) John domain $G^{\textup{core}}$
satisfying conditions (A)--(D) stated in Section~\ref{s.outline}.

We choose a sequence $\{\varepsilon_j\}_{j\in\Z}$, $\varepsilon_j\in(0,1/2)$, such that
\begin{equation}\label{e.basic_whitney}
\sum_{j=-\infty}^\infty \varepsilon_j2^{jsp}
\sum_{Q\in\mathcal{W}_j(\R^2\setminus \bar G^{\textup{core}})} \lvert Q\rvert < \infty\,.
\end{equation}
This is possible since $G^{\textup{core}}$ is bounded, and hence the inner sum is
finite for each $j$.
Then, we let $\varepsilon_{j,k}:=\varepsilon_j 2^{-k}/4$ for each
$j\in\Z$ and $k\ge 0$.
Observe, in particular, that
\begin{equation}\label{e.series}
\sum_{k'=0}^\infty \varepsilon_{j,k'} = \varepsilon_j/2<1/4\,,\qquad j\in\Z\,.
\end{equation}
Define a closed set that is parameterized by $\{\varepsilon_{j,k}\}$
and has no interior points,
\[
K:=\partial G^{\textup{core}} \cup \bigcup \{x\,:\,x\in K_j(Q)\text{ for some }
Q\in\mathcal{W}_j(\R^2\setminus \bar G^{\textup{core}})\text{ and } j\in\Z\}\,.
\]
Here the compact sets $K_j(Q)$ are as defined in Section~\ref{s.aux_compact}.
Observe that it is important to include the boundary of $G^{\textup{core}}$ to the
union above, as
otherwise $K$ would not be closed.
Finally, we define the desired open set to be
$G := \R^2\setminus K$.
Note that $\partial G=K$ since the closed set $K$ does not contain
any interior points.

\medskip
The proof of the first claim in Theorem~\ref{t.counter} is finished by
Proposition~\ref{l.suffices} and the following result.

\begin{proposition}
The open set $G=\R^2\setminus K$ satisfies conditions (1)--(4).
\end{proposition}

\begin{proof}
By construction, conditions (1) and (2) are satisfied.
Let us then consider condition (3).
Observe that
$\lvert Q\cap G\rvert = \lvert Q\setminus K_j(Q)\rvert $ if $Q\in\mathcal{W}_j(\R^2\setminus \bar G^{\textup{core}})$. Hence,
by Lemma \ref{l.diff} (recall that $Q_{j,0,1}=Q_j=Q$ therein),
\begin{align*}
\sum_{j=-\infty}^\infty 2^{jsp}\sum_{Q\in\mathcal{W}_j(\R^2\setminus \bar G^{\textup{core}})}\lvert Q\cap G\rvert
&\le \sum_{j=-\infty}^\infty 2^{jsp}\sum_{Q\in\mathcal{W}_j(\R^2\setminus \bar G^{\textup{core}})}
  \lvert Q\rvert\sum_{k'=0}^\infty 2\varepsilon_{j,k'}\,.
\end{align*}
Condition (3) follows from this by using inequalities~\eqref{e.series} and~\eqref{e.basic_whitney}.
In order to verify condition~(4), it suffices to show that $\partial G=K$
is an unbounded $2$-regular set.
Indeed, Proposition~\ref{t.lambda_sets_are_fat} then implies that
$\partial G$ is $(s,p)$-uniformly fat.

For a fixed $x\in K$ and $0<r<\infty$, we have either (a) or (b) below:

{\bf (a)}: $\dist(x,\partial G^{\textup{core}}) <  r/2$.
By co-plumpness, condition (D), there is $Q\in \mathcal{W}(\R^2\setminus \bar G^{\textup{core}})$ such that
$\lvert Q\rvert \ge \eta r^2$ and $Q\subset B(x,r)$. By Lemma \ref{l.diff}
and inequality \eqref{e.series},
\[
\lvert B(x,r)\cap K\rvert \ge \lvert K_j(Q)\rvert \ge r^2 \eta/2\,.
\]

{\bf (b)}: $r/2\le \dist(x,\partial G^{\textup{core}})$.
In this case, $x\in K_j(Q)\subset Q$ for $Q\in\mathcal{W}_j(\R^2\setminus \bar G^{\textup{core}})$,
$j\in \Z$.
Thus, by condition (C) and properties of Whitney cubes,
\[
r/2\le \dist(x,\partial (\R^2\setminus \bar G^{\textup{core}})) \le 5\diam(Q)=5\sqrt 2\cdot 2^{-j}\,.
\]
As a conclusion $r \le 10\sqrt 2\cdot 2^{-j}$, and so, by Remark \ref{r.n_set} and inequality \eqref{e.series},
\[
\lvert B(x,r)\cap K\rvert \ge \lvert B(x,r)\cap K_j(Q)\rvert \ge r^2(10\sqrt 2)^{-2}/2\,.
\]
This concludes the proof of condition (4).
\end{proof}

\subsection{A local construction}\label{s.local halfspace}

A drawback in the open set $G=\R^2\setminus K$ of the previous construction is that
$G$ is not connected.
In the following Subsections~\ref{s.local cantor} and~\ref{s.local sf}, we indicate alternative counterexamples where
$G$ is indeed a domain, thus showing that the non-connectivity is not an issue in 
the failure of the $(s,p)$-Hardy inequalities in the construction of Section~\ref{s.outline}.
The examples in Sections~\ref{s.local cantor} and~\ref{s.local sf} are based on the use of `localized' test functions
constructed in this subsection. Again, the examples are given only for the planar case, but
similar constructions can be carried out, for $0<sp\le 1$
(corresponding to $n-1\le \lambda <n$), in higher dimensions as well.

For $1\le \lambda <2$, let $K_\lambda\sub\{(x_1,x_2)\in\R^2 : x_2\ge 0\}$ be a $\lambda$-dimensional 
von Koch -type snowflake curve joining
the points $(-3,0)$ and $(3,0)$ in the plane.
For instance, when $\lambda=1$, we simply have the interval
from $(-3,0)$ to $(3,0)$, and  
the von Koch snowflake curve in Figure~1 
is obtained by joining together three copies of the curve $K_\lambda$, with $\lambda =\log 4/\log 3$, 
to create a closed curve.

Let $G\sub\R^2$ be any domain such that 
\[
K_\lambda = \bdry G \cap \big([-3,3]\times[0,3]\big)
\]
and $[-3,3]\times[0,-1] \sub G^c$; note that by the connectivity of $G$, then also the set bounded
by $K_\lambda$ and the interval from $(-3,0)$ to $(3,0)$ belongs to $G^c$. 
We demonstrate below that such a domain $G$ can not admit $(s,p)$-Hardy inequalities when $sp=2-\lambda$.
This is not yet enough to produce the desired counterexamples,
but in Sections~\ref{s.local cantor} and~\ref{s.local sf} we explain how to modify these domains
so that we obtain a proof for the second claim in Theorem~\ref{t.counter}.
Nevertheless, the construction below has also independent interest. 
In particular, when $G$ is chosen to be the upper half-space, with $\lambda = 1$,
we see that the case (T3) of Theorem~1.1 in~\cite{Dyda} is sharp:
$G$ is a domain above the graph of a Lipschitz function $\R^{n-1}\to\R$,
and $G$ does not admit $(s,p)$-Hardy inequalities when $sp=1$.

For $j,m\in\N$, we define
\begin{align*}
N_j&=\bigg\{x\in [-1,1]\times [0,2]\,:\, 2^{-j}\le \dist(x,\bdry G) < 2^{-j+1}\bigg\},\\
A_m&=\bigg\{x\in [-1,1]\times [0,2]\,:\, 2^{-m}<\dist(x,\bdry G) \le 1\bigg\} = \bigcup_{j=1}^{m} N_j,\\
E_m&=\bigg\{x\in [-2,2] \times [0,2]\,:\,0<\dist(x,\bdry G) \le 2^{-m}\bigg\}.
\end{align*}
We choose functions $u_m\in C^\infty_0(G)$ such that:
\begin{itemize}
\item[(1)] $u_m=1$ in $A_m$;

\item[(2)] $u_m=0$ outside $[-2,2]\times[0,2]$;

\item[(3)] $0\le u_m\le 1$ everywhere;

\item[(4)] $|\nabla u_m|\le c2^{m}$ everywhere and $|\nabla u_m|\le c$ in $G\setminus E_m$.
\end{itemize}
Let $1<p<\infty$ and $0<s<1$ be such that $sp=2-\lambda$.
Since  $u_m=1$ in $A_m$ and $|N_j|\simeq 2^{j(\lambda-2)}$ 
by the $\lambda$-regularity of $K_\lambda$, it follows that
\begin{equation}\label{eq:lhs=m}
\int_G \frac{\lvert u_m(x)\rvert^p}{\dist(x,\partial G)^{sp}}\,dx\ge
\sum_{j=1}^m\int_{N_j} \frac{1}{\dist(x,\partial G)^{sp}}\,dx
\simeq \sum_{j=1}^m 2^{j(\lambda-2)} 2^{jsp} = m.
\end{equation}

Hence, it suffices to show that the right-hand side of the $(s,p)$-Hardy inequality
remains uniformly bounded in $m$. 
We do this using arguments similar to those in Dyda's example~\cite{Dyda}, but due to the
local nature of our example, we have additional terms to estimate in our calculations.
First, we observe that
\begin{align*}
\int_G \int_G & \frac{\lvert u_m(x)-u_m(y)\rvert^p}{\lvert x-y\rvert^{2+sp}} \,dx\,dy\\
&\le \bigg\{2\int_{G}\int_{E_m}+\int_{G\setminus E_m}\int_{G\setminus E_m}\bigg\}\frac{\lvert u_m(x)-u_m(y)\rvert^p}
 {\lvert x-y\rvert^{2+sp}} \,dx\,dy  =: 2I_1 + I_2.
\end{align*}
By the $\lambda$-regularity of the snowflake curves,
it is easy to see that $E_m\subset \bigcup_{k=1}^{N_m} B_k$, where $B_k=B(x_k,2^{-m+1})$
for some $x_k\in K_\lambda$,
and $N=N_m\le c 2^{\lambda m}$.
Thus, we obtain
\begin{align*}
I_1 & 
  \le c\sum_{k=1}^N \int_{G}\int_{B_k}\frac{\lvert u_m(x)-u_m(y)\rvert^p}{\lvert x-y\rvert^{2+sp}} \,dx\,dy \\
 &\le c\sum_{\ell=0}^{\infty}\sum_{k=1}^N \int_{(\ell+1)B_k\setminus \ell B_k}\int_{B_k}
           \frac{\lvert u_m(x)-u_m(y)\rvert^p}{\lvert x-y\rvert^{2+sp}} \,dx\,dy.
\end{align*}
Using property~(4) of the functions $u_m$, the terms for $\ell=0$ and $\ell=1$ in the above sum are bounded by 
\begin{align*}
\sum_{k=1}^N \int_{2B_k}\int_{B_k} & \frac{2^{mp}|x-y|^p}{\lvert x-y\rvert^{2+sp}} \,dx\,dy 
 \le \sum_{k=1}^N \int_{2B_k}\int_{B(y,2^{-m+3})} 2^{mp}{\lvert x-y\rvert^{p-2-sp}} \,dx\,dy \\ 
& \le c \sum_{k=1}^N |B_k| 2^{mp} 2^{-m(p-sp)} 
  \le c 2^{\lambda m} 2^{-2m}  2^{mp} 2^{-m(p-sp)}\le c;
\end{align*}
recall here that $sp=2-\lambda$. For $\ell=2,3,\dots$, we have
\begin{align*}
\sum_{k=1}^N & \int_{(\ell+1)B_k\setminus \ell B_k}\int_{B_k} 
           \frac{\lvert u_m(x)-u_m(y)\rvert^p}{\lvert x-y\rvert^{2+sp}} \,dx\,dy\\  
   & \le  c \sum_{k=1}^N |(\ell+1)B_k\setminus \ell B_k|\cdot|B_k|\frac{1}{(\ell 2^{-m})^{2+sp}}\\
   & \le  c 2^{\lambda m} (\ell 2^{-m})2^{-m} |B_k| 2^{m(2+sp)} \frac{1}{\ell^{2+sp}}
     \le c  2^{m(\lambda-2-2+2+sp)}\ell^{-(1+sp)} \le c\,\ell^{-(1+sp)}. 
\end{align*}
Since $sp>0$, it follows that
\[
I_1\le c \sum_{\ell=0}^{\infty}\sum_{k=1}^N \int_{(\ell+1)B_k\setminus \ell B_k}\int_{B_k}
           \frac{\lvert u_m(x)-u_m(y)\rvert^p}{\lvert x-y\rvert^{2+sp}} \,dx\,dy
   \le c + \sum_{\ell=2}^{\infty} \ell^{-(1+sp)} \le c.
\]

Let us then estimate the integral $I_2$. We write
\[
Q_m^2=\big(([-2,2]\times [0,2])\setminus E_m\big)\cap G\quad\text{ and }\quad 
Q_m^3=\big(([-3,3]\times [0,3])\setminus E_m\big)\cap G.
\]
Then, with simple manipulations and the fact that  $u_m=0$  outside  of the set $[-2,2]\times[0,2]$, we find
\[\begin{split}
I_2 & = \int_{G\setminus E_m}\int_{G\setminus E_m} \frac{\lvert u_m(x)-u_m(y)\rvert^p}{\lvert x-y\rvert^{2+sp}} \,dx\,dy\\
 & \le \bigg\{\int_{Q^3_m}\int_{Q^2_m} + 2 \int_{Q^2_m}\int_{G\setminus (Q^3_m\cup E_m)}\bigg\}
     \frac{\lvert u_m(x)-u_m(y)\rvert^p}{\lvert x-y\rvert^{2+sp}} \,dx\,dy
 =: I_{21} + 2I_{22}.
\end{split}\]
By the properties of snowflake curves,
the set $Q^3_m$ is quasiconvex; that is, all pairs $x,y\in Q^3_m$ can be joined by a curve $\gamma_{x,y}\sub Q^3_m$
such that $\length(\gamma)\le c|x-y|$. Thus $|u_m(x)-u_m(y)|\le c |x-y|$ for all $x,y\in Q^3_m$
by property~(4) of the functions $u_m$, and so
\[
I_{21} \le c \int_{Q^3_m}\int_{Q^2_m}\frac{\lvert x-y \rvert^p}{\lvert x-y\rvert^{2+sp}} \,dx\,dy
\le c \int_{Q^3_m}\int_{B(y,10)}\lvert x-y \rvert^{p-2-sp}\,dx\,dy\le c.
\]
Finally,
\begin{equation}\label{eq:I22}
\begin{split}
I_{22} & = \int_{Q^2_m}\int_{G\setminus (Q^3_m\cup E_m)}
   \frac{\lvert u_m(x)-u_m(y)\rvert^p}{\lvert x-y\rvert^{2+sp}}\,dx\,dy\\
 & \le\int_{Q^2_m}\int_{G\setminus B(y,1)}\frac{1}{\lvert x-y\rvert^{2+sp}} \,dx\,dy \le c,
\end{split}
\end{equation}
where we used also the fact that $[-3,3]\times[0,-1] \sub G^c$.
By the above calculations we conclude that, indeed,
\[
\int_G \int_G \frac{\lvert u_m(x)-u_m(y)\rvert^p}{\lvert x-y\rvert^{2+sp}} \,dx\,dy \le C,
\]
where the constant $C>0$ is independent of $m$. Consequently, $G$ can not admit
$(s,p)$-Hardy inequalities for $sp=2-\lambda$.

\subsection{Local counterexample for $sp=1$}\label{s.local cantor}

Let us now consider a domain $G\sub \R^2$ such that $G=G^{\textup{core}}\setminus K$,
where $G^{\textup{core}}$ is a snowflake domain with $\dim_\Ha(\bdry G^{\textup{core}})=\lambda\in(1,2)$
(cf.\ Section~\ref{s.sel}) and $K$ is a fat Cantor set (cf.\ Section~\ref{s.aux_compact}) placed inside
$G^{\textup{core}}$ in such a way that $\dist(K,\bdry G^{\textup{core}})\ge c_0>0$.
Then the boundary of $G$ is locally $(s,p)$-uniformly fat whenever $sp>2-\lambda$,
and so in particular for $sp=1$, but we will show that
$G$ can not admit $(s,p)$-Hardy inequalities when $sp=1$.

For simplicity, we may assume that $K=K(Q_0)$ with $Q_0=[0,1]\times [0,1]$, as constructed in
Section~\ref{s.sel}, and that $\dist(K,\bdry G^{\textup{core}})\ge 1$. For each $k\ge 0$, let $Q_{k}=Q_{0,k,m}$
be the $k$th level cube in the construction of $K$ (with side length $\ell_k=\ell_{0,k}$)
satisfying $(0,1)\in Q_{k}$. The idea is to place 
above each $Q_{k}$ a scaled copy of a suitable function $u_m$, $m=m(k)$, from Section~\ref{s.local halfspace}
with $\lambda = 1$.
More precisely, the cube $[0,\ell_{k}]\times[1,1+\ell_{k}/2]$ corresponds now to the cube $[-3,3]\times [0,3]$\,;
the line segment from $(0,1)$ to $(0+\ell_k,1)$ corresponds to $K_1$ of Section~\ref{s.local halfspace}; 
and the functions $u_m$ are scaled accordingly.
If $2^{-m}\ge \varepsilon_{k}$, where $\varepsilon_{k}=\varepsilon_{0,k}$ gives the relative width of
the `corridor' in the $(k+1)$th level of the construction of $K$, then estimate~\eqref{eq:lhs=m}, when applied to these
scaled copies of $u_m$ in $[0,\ell_{k}]\times[1,1+\ell_{k}/2]$, yields
\begin{equation}\label{eq:lhs=m*}
\int_G \frac{\lvert u_m(x)\rvert^p}{\dist(x,\partial G)^{sp}}\,dx\ge c \ell_k^{2-sp} m.
\end{equation}
Notice, however, that for a fixed $k$ this estimate is not true when $m\to\infty$, since then also the
small corridors in $Q_k\setminus K$ influence the factor $\dist(x,\bdry G)$ in the integral. 

The terms on the right-hand side of the $(s,p)$-Hardy inequality are scaled with the
factor $\ell_k^{2-sp}$, too. Estimates for the integrals 
corresponding to $I_1$ and $I_{21}$ from Section~\ref{s.local halfspace}
are valid as such, and thus these terms are bounded by $c \ell_k^{2-sp}$. However, in estimate~\eqref{eq:I22}
for the integral $I_{22}$ we used the fact that $[-3,3]\times[0,-1] \sub G^c$, but the
corresponding claim is clearly not true in the present setting; in fact, $Q_k\cap G$ is even a dense set.
To overcome this, we need to be a bit more careful when estimating this last integral.

We choose $\varepsilon_k=2^{-k(2+sp)}/4$, and thus $\sum_{k'=k}^\infty 2\varepsilon_{k'}\simeq \varepsilon_k$ for all $k\ge 0$.
By Lemma~\ref{l.diff}, we know that then
\begin{equation}\label{eq:small area}
|Q_k\cap G|\le |Q_k|\sum_{k'=k}^\infty 2\varepsilon_{k'} \simeq \ell_k^2 \varepsilon_{k} \simeq \ell_k^2 2^{-k(2+sp)}.
\end{equation}
Let $\hat Q^2_m,\hat Q^3_m\sub[0,\ell_{k}]\times[1,1+\ell_{k}/2]$ 
denote the obvious analogs of $Q^2_m$ and $Q^3_m$.
Then we can estimate the integral corresponding to $I_{22}$ as
\[\begin{split}
\int_{\hat Q^2_m} & \int_{G\setminus \hat Q^3_m}
   \frac{\lvert u_m(x)-u_m(y)\rvert^p}{\lvert x-y\rvert^{2+sp}}\,dx\,dy\\
& \le\int_{\hat Q^2_m}\int_{G\setminus B(y,\ell_k/6)}\frac{1}{\lvert x-y\rvert^{2+sp}} \,dx\,dy
    + \int_{\hat Q^2_m}\int_{Q_k\cap G}\frac{1}{\lvert x-y\rvert^{2+sp}} \,dx\,dy\\
& \le c \ell_k^{2-sp} + |Q_k||Q_k\cap G|2^{m(2+sp)}\ell_k^{-(2+sp)} \\
& \le c \ell_k^{2-sp} + \ell_k^{2-sp} 2^{-k(2+sp)} 2^{m(2+sp)},
\end{split}\]
where we used~\eqref{eq:small area} and the fact that $|x-y|\ge \ell_k 2^{-m}$ 
when $y\in \hat Q^2_m$ and $x\in Q_k\cap G$. Thus, for $m=m(k)=k$, 
the integral above is bounded by $\ell_k^{2-sp}$, as desired. Moreover,
with this choice we have $2^{-m(k)}=2^{-k}\ge 2^{-k(2+sp)} \ge \varepsilon_k$, which
was the only requirement for estimate~\eqref{eq:lhs=m*} to hold.
In conclusion, the $(s,p)$-Hardy inequality can not be valid for functions $u_{m(k)}$, $m(k)=k$,
with a uniform constant, and hence the domain $G$ can not admit $(s,p)$-Hardy inequalities
when $sp=1$.

\subsection{Local snowflaked counterexamples for $0<sp<1$}\label{s.local sf}

Let us briefly indicate how to obtain similar counterexamples when $0<sp<1$.
The value $sp=1$ above corresponds to the fact that the fat Cantor set is,
loosely speaking, seen as one-dimensional when observed from a suitable distance from
within the domain; here $1=2-sp$. Similarly, if we had a locally $(s,p)$-uniformly 
fat part of the boundary that is seen as $\lambda$-dimensional,
with $\lambda = 2 -sp$, 
then $(s,p)$-Hardy inequalities would fail. Such a situation can be obtained, for instance,
by placing a snowflake type set $F_\lambda$, indicated in 
Figure~3,
inside a suitable $G^{\textup{core}}$-domain. Notice, however, that the
set in Figure~3 is based on a very simple approximation of the
actual snowflake curve $K_\lambda$.

\begin{figure}[!ht]
\begin{center}
\includegraphics[width=12cm]{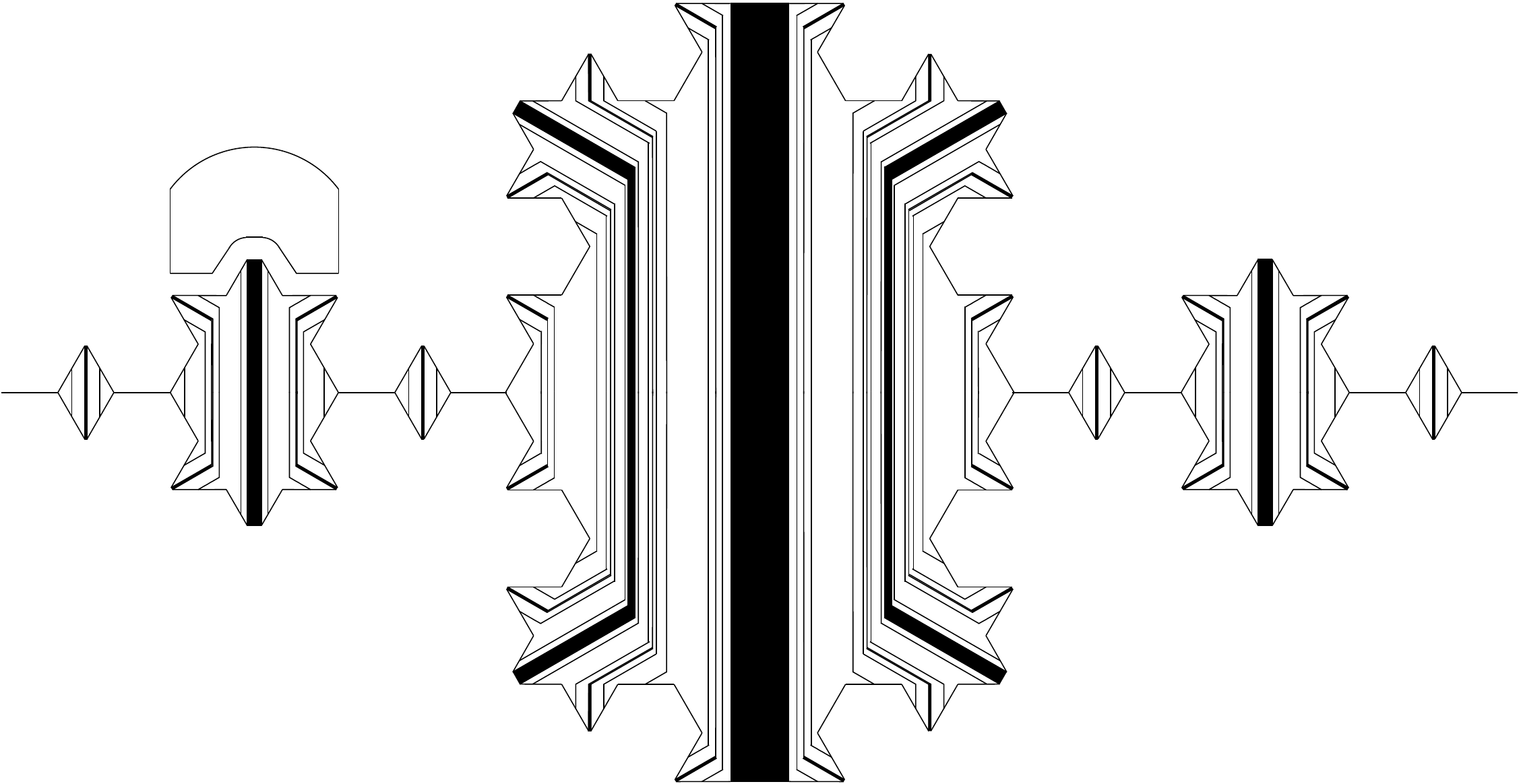}
\end{center}
\caption{Snowflake set $F_\lambda$ with tunnels. The support of the test function $u_2$
is drawn above the snowflake rhombus $Q_2$.}
\end{figure}

More precisely, to obtain such a set $F_\lambda$, 
we first consider the set bounded by a snowflake curve $K_\lambda$ of dimension $\lambda\in (1,2)$ and its
reflection on the $x_1$-axis, and then remove `tunnels' 
(the black parts in Figure~3)
of width $\varepsilon_k$: one large tunnel of width $\varepsilon_1$, four tunnels of width $\varepsilon_2$, etc.
It is clear that if we choose $\varepsilon_k>0$ to be small enough, then the resulting set $F_\lambda$
is still $2$-regular, and thus locally $(s,p)$-uniformly fat for any $1<p<\infty$ and $0<s<1$. 
If the boundary of the
surrounding $G^{\textup{core}}$-domain has dimension greater than $\lambda$, 
it follows that the boundary
of $G=G^{\textup{core}}\setminus F_\lambda$ is locally $(s,p)$-uniformly fat for $sp=2-\lambda$.

Now, let $Q_1\sub \partial G$ denote the large `snowflake rhombus' in the middle of $K_\lambda$.
Then choose $Q_2$ to be the second largest rhombus on the
left-hand side of $Q_1$; 
and $Q_3$ to be the third largest rhombus on the
left-hand side of $Q_2$; 
and so on. Just like in the counterexample 
of Section~\ref{s.local cantor} with the fat Cantor set, we place scaled copies of the functions $u_m$ from 
Section~\ref{s.local halfspace} (constructed with respect to the corresponding $K_\lambda$) above the
sets $Q_k$; cf.~Figure~3. With computations analogous to those in Section~\ref{s.local cantor} 
(we leave the details 
to the reader), it follows that, for $sp=2-\lambda$, $G$ can not admit 
$(s,p)$-Hardy inequalities, and yet $G$ is a domain
having a locally $(s,p)$-uniformly fat boundary. This finishes the proof of
the second claim of Theorem~\ref{t.counter}.

\section{Applications}\label{s.applications}

We conclude the paper with extension and removability problems related to
fractional Hardy inequalities and fractional Sobolev spaces. 
For $0<s<1$ and  $1< p <\infty$, the \emph{fractional Sobolev space} $W^{s,p}(G)$ 
is defined as the space of functions $u\in L^p({G})$ with
$\lVert u\rVert_{W^{s,p}({G})}:=\lVert u\rVert_{L^p({G})}+|u|_{W^{s,p}({G})}<\infty$,
where
\begin{equation}\label{d.frac_def}
|u|_{W^{s,p}({G})} := \bigg(\int_{G} \int_{G}
\frac{|u(x)-u(y)|^p}{|x-y|^{n+sp}}\,dy\,dx\bigg)^{1/p}.
\end{equation}
See \cite{HH} for a nice introduction and basic properties of the spaces $W^{s,p}(G)$.

\subsection{Zero Extension}
We recall here a straightforward but useful connection between
boundedness of certain extension operators and Hardy inequalities~\cite{ihnatsyeva2}. 
For an open set $G\subset\R^n$, we define the
`zero extension operator' $E_G$ as
\begin{equation}\label{zero_ext_def}
E_G\colon L^{p}(G)\to L^{p}(\R^n)\,,\quad E_G u(x)
=\begin{cases} u(x)\,,\qquad &x\in G\,;\\ 0\,,\qquad &x\in \R^n\setminus G\,.\end{cases}
\end{equation}
We say that the operator $E_G$ is 
bounded, 
if there is a constant $c_1>0$ such that
\[
\lvert E_G u\rvert_{W^{s,p}(\R^n)}^p
\le c_1 \lvert u\rvert_{W^{s,p}(G)}^p
\]
for every $u\in C^\infty_0(G)$.

\begin{lemma}\label{l.chi_bounded}
Let $0<s<1$ and $1<p<\infty$, and suppose that $G\subset\R^n$ is an open set.
Then the zero extension $E_G u$ of any $u\in W^{s,p}(G)$ satisfies the inequality
\[
\lvert E_G u \rvert_{W^{s,p}(\R^n)}^p
\le \lvert u\rvert_{W^{s,p}(G)}^p + C \int_{{G}}
\frac{\lvert u(x)\rvert^p}{ \mathrm{dist}(x,\partial G)^{sp}}\,dx\,,
\]
where the constant $C$ depends on $n$, $s$, and $p$.
\end{lemma}

\begin{proof}
Since
\begin{align*}
\lvert E_G u \rvert_{W^{s,p}(\R^n)}^p &=
\int_{\R^n}\int_{\R^n}\frac{\lvert E_Gu(x) - E_Gu(y)\rvert^p}
{\lvert x-y\rvert^{n+sp}}\,dy\,dx\\
&\le \lvert u\rvert_{W^{s,p}(G)}^p + 2 \int_{{G}}
\lvert u(x)\rvert^p \int_{G^c} \frac{1}{\lvert x-y\rvert^{n+sp}}\,dy\,dx\,,
\end{align*}
it remains to apply the following estimates, which are valid
for $x\in{G}$:
\begin{equation}\label{e.integral}
\begin{split}
\int_{G^c} \frac{1}{\lvert x-y\rvert^{n+sp}}\,dy
&\le \int_{\R^n\setminus B(x,\mathrm{dist}(x,\partial{G}))} \frac{1}{\lvert x-y\rvert^{n+sp}}\,dy\\
& \le C \int_{\mathrm{dist}(x,\partial G)}^\infty r^{-1-sp}\,dr \le C
\mathrm{dist}(x,\partial{G})^{-sp}\,.
\end{split}
\end{equation}
This completes the proof of the lemma.
\end{proof}

We also need the following auxiliary Hardy-type inequality, an
adaptation of results in~\cite{E-HSV} and~\cite[Theorem~3.1]{ihnatsyeva1}.

\begin{theorem}\label{fractional_hardy_fat}
Let $1<p<\infty$  and $0<s<1$ be such that $sp<n$.
Let $G\subset\R^n $ be an open set, $n\geq 2$, whose
complement is $(s,p)$-uniformly fat.
Then, for every $u\in C^\infty_0(G)$,
\[
\int_{G}\frac{\vert u(x)\vert^p}{\dist (x,\partial G)^{sp} }\,dx
\le c_0\int_{\R^n}\int_{\R^n}\frac{\vert E_G u(x)-E_G u(y)\vert^p}{\vert x-y\vert ^{n+sp} }\,dy\,dx\,,
\]
where 
the constant $c_0$ depends on $s$, $n$, $p$, and the fatness constant $\sigma$.
\end{theorem}

With the help of the above results, we establish a connection between
extensions and Hardy inequalities:

\begin{theorem}\label{equiv}
Let $1<p<\infty$ and $0<s<1$ be such that $sp<n$.
Let $G\subset\R^n$, $n\geq 2$, be an open set.
Then the following two statements hold:
\begin{itemize}
\item[(1)]
If $G^c$ is $(s,p)$-uniformly fat and $E_G$ is bounded, then $G$ admits an $(s,p)$-Hardy inequality.
\item[(2)]
Conversely, suppose that $G$ admits an $(s,p)$-Hardy inequality. 
Then $E_G$ is bounded.
\end{itemize}
\end{theorem}

\begin{proof}
First suppose that the assumptions in (1) hold. By Theorem \ref{fractional_hardy_fat} we have 
for every $u\in C^\infty_0(G)$ that
\begin{align*}
\int_{G}\frac{\vert u(x)\vert^p}{\dist (x,\partial G)^{sp} }\,dx
\le c_0 \lvert E_G u\rvert_{W^{s,p}(\R^n)}^p \le c_0c_1\lvert u\rvert_{W^{s,p}(G)}^p\,.
\end{align*}
Hence, the $(s,p)$-Hardy inequality holds with constant $c_2=c_0c_1$.

Suppose then that the assumption in (2) holds;
in other words, there is $c_2>0$ such that, for every $u\in C^\infty_0(G)$,
\[
\int_{G}\frac{\lvert u(x)\rvert^p}{\dist(x,\partial G)^{sp}}\,dx \le c_2\lvert u\rvert_{W^{s,p}(G)}^p\,.
\]
By Lemma \ref{l.chi_bounded} we then have
\begin{align*}
\lvert E_G u\rvert_{W^{s,p}(\R^n)}^p
\le \lvert u\rvert_{W^{s,p}(G)}^p + C \int_G \frac{\lvert u(x)\rvert^p}{\dist(x,\partial G)^{sp}}\,dx
\le (1+C c_2) \lvert u\rvert_{W^{s,p}(G)}^p\,
\end{align*}
for every $u\in C^\infty_0(G)$,
and thus $E_G$ is bounded.
\end{proof}

The following corollary is an immediate consequence of Corollary~\ref{cor:fractional_hardy_uniform} 
and Theorem~\ref{equiv}.

\begin{corollary}
Let $1<p<\infty$ and $0<s<1$ be such that $sp<n$.
Suppose that $G\subset \R^n$ is a (bounded) uniform domain whose boundary is 
(locally) $(s,p)$-uniformly fat.
Then $E_G$ is bounded.
\end{corollary}

This corollary applies for instance to the usual Koch snowflake domain $G$
with $2-sp<\log 4/ \log 3$.

\subsection{Removability}\label{sect:remove}

Finally, let us record the following removability result for fractional Hardy inequalities.

\begin{theorem}\label{removal}
Let $1<p<\infty$ and $0<s<1$ be such that $sp<n$.
Assume that $G'\subset G\subset\R^n$ are open sets such that
$(G')^c$ is $(s,p)$-uniformly fat, $G$ admits an $(s,p)$-Hardy inequality,
and $|G\setminus G'|=0$.
Then $G'$ admits an $(s,p)$-Hardy inequality.
\end{theorem}

\begin{proof}
By Theorem \ref{equiv}(2), there is a constant $c_1>0$ such that, for every $g\in C^\infty_0(G)$,
\begin{equation}\label{ext}
\lvert E_G g\rvert^p_{W^{s,p}(\R^n)} \le c_1 \lvert g\rvert^p_{W^{s,p}(G)}\,.
\end{equation}
Let $u\in C^\infty_0(G')$, and define
$g:=(E_{G'} u)\lvert_G\in C^\infty_0(G)$. An application of the previous inequality~\eqref{ext} to this
function yields
\begin{align*}
\lvert E_{G'} u\rvert^p_{W^{s,p}(\R^n)}
& =\lvert E_G g\rvert_{W^{s,p}(\R^n)}^p  \le c_1 \lvert g\rvert_{W^{s,p}(G)}^p
   = c_1 \big\lvert (E_{G'} u)\lvert_G \big\rvert_{W^{s,p}(G)}^p\\
&=c_1\int_G \int_G \frac{\lvert E_{G'} u(x) - E_{G'} u(y)\lvert ^p}{\lvert x-y\rvert^{n+sp}}\\
&=c_1\int_{G'} \int_{G'} \frac{\lvert u(x) - u(y)\lvert ^p}{\lvert x-y\rvert^{n+sp}}
=c_1\lvert u\rvert_{W^{s,p}(G')}^p\,.
\end{align*}
In the penultimate step, we used the assumption that $|G\setminus G'|=0$.
As a conclusion, we have shown that
$\lvert E_{G'} u\rvert^p_{W^{s,p}(\R^n)}  \le c_1\lvert u\rvert_{W^{s,p}(G')}^p$
for every $u\in C^\infty_0(G')$. By Theorem~\ref{equiv}, we find
that $G'$ admits an $(s,p)$-Hardy inequality.
\end{proof}

\begin{remark}
Let us indicate a possible application of
Theorem \ref{removal}.
By Corollary \ref{cor:fractional_hardy_uniform}, we  may begin with a uniform domain $G\subset \R^n$ whose boundary
is $(s,p)$-uniformly fat with  $1<p<\infty$ and $0<sp<n$.
Let $K\subset G$ be a closed $(s,p)$-uniformly fat set with zero measure. Then
it follows from Theorem~\ref{removal} that
the open set $G':=G\setminus K$ admits
an $(s,p)$-Hardy inequality.
\end{remark}

\subsection*{Acknowledgements}
The research is supported  by
the Academy of Finland, grants no.\ 135561 (L.I. and H.T.) and 252108 (J.L.).


\begin{thebibliography}{999}






\normalsize
\baselineskip=17pt



\bibitem{AH} D. R. Adams and L. I. Hedberg, 
\emph{Function spaces and potential theory},
Grundlehren Math. Wiss.  314,
Springer-Verlag, Berlin, 1996.

\bibitem{BogdanDyda}
K. Bogdan and B. Dyda,
\emph{The best constant in a fractional order Hardy inequality},
{Math. Nachr.} {284}  (2011),
{629--638}.


\bibitem{Carleson}
L. Carleson, 
\emph{Selected problems on exceptional sets},
Van Nostrand Mathematical Studies, No. 13
D. Van Nostrand Co., Inc., Princeton, N.J.-Toronto, Ont.-London, 1967.


\bibitem{HH}
E. Di Nezza, G. Palatucci, and E. Valdinoci,
\emph{Hitchhiker's guide to the fractional Sobolev spaces}, 
Bull. Sci. Math. {136} (2012), 521--573. 

\bibitem{Dyda}
B. Dyda, \emph{A fractional order Hardy inequality}, 
Illinois J. Math. {48} (2004), 575--588.

\bibitem{Dyda2}
B. Dyda, \emph{Fractional Hardy inequality with a remainder term},
Colloq. Math. {122} (2011), 59--67.

\bibitem{DydaFrank}
B. Dyda and R. L. Frank,
\emph{Fractional Hardy--Sobolev--Maz'ya inequality for domains},
Studia Math. {208} (2012), 151--166.

\bibitem{DV}
B. Dyda and {A. V. V{\"a}h{\"a}kangas}, 
\emph{A framework for fractional Hardy inequalities},
Ann. Acad. Sci. Fenn. Math. 39 (2014), 675--689.

\bibitem{E-HSV}
{D. E. Edmunds}, {R. Hurri-Syrj\"anen} and {A. V. V{\"a}h{\"a}kangas},
\emph{Fractional Hardy-type inequalities in domains with uniformly fat complement},
{Proc. Amer. Math. Soc.} 142 (2014), 897--907.

\bibitem{FMT}
S. Filippas, L. Moschini, and A. Tertikas,
\emph{Sharp trace Hardy--Sobolev--Maz'ya inequalities and the fractional Laplacian},
Arch. Ration. Mech. Anal. {208} (2013), 109--161.

\bibitem{FrankSeiringer}
R. L. Frank and R. Seiringer,
\emph{Sharp Fractional Hardy Inequalities in Half-Spaces},
Around the research of Vladimir Maz'ya. I, 161--167,
Int. Math. Ser. (N. Y.), 11, Springer, New York, 2010. 


\bibitem{H-SV}
{R. Hurri-Syrj\"anen} and {A. V. V{\"a}h{\"a}kangas},
\emph{On fractional Poincar\'e inequalities}, 
J. Anal. Math. {120} (2013), 85--104.

\bibitem{ihnatsyeva1}
{L. Ihnatsyeva} and {A. V. V\"ah\"akangas},
\emph{Hardy inequalities in Triebel--Lizorkin spaces}, 
Indiana Univ. Math. J. 62 (2013), 1785--1807.


\bibitem{ihnatsyeva2}
{L. Ihnatsyeva} and {A. V. V\"ah\"akangas},
\emph{Hardy inequalities in Triebel--Lizorkin spaces II. Aikawa dimension},  
Ann. Mat. Pura Appl. (4) (2013), DOI:10.1007/s10231-013-0385-z.


\bibitem{KLT}
R. Korte, J. Lehrb\"ack and H. Tuominen,
\emph{The equivalence between pointwise Hardy inequalities and uniform fatness}, 
Math. Ann. 351 (2011), 711--731.

\bibitem{KOSKELA}
P. Koskela,
\emph{The degree of regularity of a quasiconformal mapping},
Proc. Amer. Math. Soc. 122 (1994), 769--772.

\bibitem{kole}
P. Koskela and J. Lehrb\"ack,
\emph{Weighted pointwise Hardy inequalities},
J.~London Math. Soc. {79} (2009), 757--779.

\bibitem{KufnerPersson}
A. Kufner and L. E. Persson,
\emph{Weighted inequalities of Hardy type},
World Scientific Publishing Co., New Jersey, 2003.

\bibitem{LePoint}
J. Lehrb\"ack, \emph{Pointwise Hardy inequalities and uniformly fat sets}, 
Proc. Amer. Math. Soc. {136} (2008), 2193--2200.


\bibitem{lehrback}
J. Lehrb\"ack,
\emph{Weighted Hardy inequalities and the size of the boundary},
Manuscripta Math. {127} (2008), 249--273.

\bibitem{LeLip}
J. Lehrb\"ack, \emph{Weighted Hardy inequalities beyond Lipschitz domains}, 
Proc. Amer. Math. Soc. 142 (2014), 1705--1715.



\bibitem{Lewis1988}
J. L. Lewis, \emph{Uniformly fat sets},
Trans. Amer. Math. Soc. {308} (1988), 177--196.

\bibitem{LossSloane2010}
M. Loss and C. A. Sloane, \emph{Hardy inequalities for fractional integrals on general domains},
J. Funct. Anal. {259} (2010), 1369--1379.

\bibitem{Mattila}
P. Mattila, \emph{Geometry of sets and measures in Euclidean spaces. Fractals and rectifiability},
Cambridge Studies in Advanced Mathematics, 44. Cambridge University Press, Cambridge, 1995.

\bibitem{Necas}
J. Ne\v cas,
\emph{Sur une m\'ethode pour r\'esoudre les \'equations aux d\'eriv\'ees
partielles du type elliptique, voisine de la variationnelle},
Ann. Scuola Norm. Sup.  Pisa (3)  {16}  (1962), 305--326.

\bibitem{Sloane}
C. A. Sloane,
\emph{A fractional Hardy--Sobolev--Maz'ya inequality on the upper halfspace},
Proc. Amer. Math. Soc. {139} (2011), 4003--4016.

\bibitem{S}
E. M. Stein,
\emph{Singular integrals and differentiability properties of functions},
Princeton Univ. Press, Princeton, New Jersey, 1970.

\bibitem{Wannebo}
A. Wannebo,
\emph{Hardy inequalities},
Proc. Amer. Math. Soc. {109} (1990), 85--95.


\end{thebibliography}
\end{document}